\documentclass[12pt]{article}
\usepackage{amsmath,amsthm,amssymb,latexsym,fancyhdr,color,mathrsfs}
\usepackage[dvipsnames]{xcolor}
\usepackage{hyperref}
\usepackage{cleveref}
\usepackage[margin=1in]{geometry}
\usepackage[utf8]{inputenc}
\usepackage{enumitem}
\theoremstyle{plain}
\newtheorem{theorem}{Theorem}
\newtheorem*{theorem*}{Theorem}

\newtheorem{lemma}[theorem]{Lemma}
\newtheorem{proposition}[theorem]{Proposition}
\newtheorem{corollary}[theorem]{Corollary}
\newtheorem{claim}[theorem]{Claim}

\theoremstyle{definition} 
\newtheorem{definition}[theorem]{Definition}
\newtheorem{remark}[theorem]{Remark}

\newcommand{\tn}{\tilde{\nabla}}
\newcommand{\firstvar}[2]{\frac{\delta {#1}}{\delta {#2}}}

\newcommand{\iR}{\int_{\mathbb{R}^d}}
\newcommand{\iiR}{\iint_{\mathbb{R}^{2d}}}
\newcommand{\Rd}{\mathbb{R}^d}
\newcommand{\R}{\mathbb{R}}

\newcommand{\textitunder}[1]{\noindent\underline{\textit{{#1}}}}

\newcommand{\PtwoE}{\mathscr{P}_{2,E}(\Rd)}

\newcommand{\japangle}[1]{\left\langle #1 \right\rangle}
\newcommand{\Rthree}{\mathbb{R}^3}
\newcommand{\iRthree}{\int_{\mathbb{R}^3}}
\newcommand{\iiRsix}{\iint_{\mathbb{R}^6}}

\title{The Landau equation as a Gradient Flow}
\author{Jos\'e A. Carrillo\thanks{Mathematical Institute, University of Oxford, Oxford OX2 6GG, UK (carrillo@maths.ox.ac.uk)}, Matias G. Delgadino\thanks{Department of Mathematics, The University of Texas at Austin, Texas, USA (matias.delgadino@math.utexas.edu)}, Laurent Desvillettes\thanks{Universit\'e Paris Cit\'e, Sorbonne Universit\'e, CNRS, IUF, Institut de Math\'ematiques de Jussieu-Paris Rive Gauche (desvillettes@math.univ-paris-diderot.fr)}, Jeremy Wu\thanks{Mathematical Sciences Building, University of California, Los Angeles, USA (jeremywu@math.ucla.edu)}}

\begin{document}
	\maketitle
	
	\begin{abstract}
		We propose a gradient flow perspective to the spatially homogeneous Landau equation for soft potentials. We construct a tailored metric on the space of probability measures based on the entropy dissipation of the Landau equation. Under this metric, the Landau equation can be characterized as the gradient flow of the Boltzmann entropy. In particular, we characterize the dynamics of the PDE through a functional inequality which is usually referred as the Energy Dissipation Inequality (EDI). Furthermore, analogous to the optimal transportation setting, we show that this interpretation can be used in a minimizing movement scheme to construct solutions to a regularized Landau equation.
	\end{abstract}

	\section{Introduction}
	\label{sec:intro}
	The Landau equation is an important partial differential equation in kinetic theory. It gives a description of colliding particles in plasma physics~\cite{lifshitz_physical_1981}, and it can be formally derived as a limit of the Boltzmann equation where grazing collisions are dominant~\cite{DLD92,villani_new_1998}. Similar to the Boltzmann equation (see \cite{BPS13} for a consistency result and related derivation issues), the rigorous derivation of the Landau equation from particle dynamics is still a huge challenge. 
	For a spatially homogeneous density of particles $f = f_t(v)$ for $t\in (0,\infty), v\in\Rd$ the homogeneous Landau equation reads
	\begin{equation}
		\label{eq:lan}
		\partial_t f(v) = \nabla_v \cdot \left(
		f(v)\int_{\mathbb{R}^d} |v-v_*|^{2+\gamma}\Pi[v-v_*](\nabla_v \log f(v) - \nabla_{v_*} \log f(v_*))f(v_*) dv_*
		\right).
	\end{equation}
	For notational convenience, we sometimes abbreviate $f = f_t(v)$ and $f_{*} = f_t(v_*)$ . We also denote the differentiations by $\nabla = \nabla_v$ and $\nabla_* = \nabla_{v_*}$. The physically relevant parameters are usually $d=2,3$ and $\gamma \geq -d-1$ with $\Pi[z] = I - \frac{z\otimes z}{|z|^2}$ being the projection matrix onto $\{ z\}^\perp$. In this paper, for simplicity we will focus in the case $d=3$ and vary the weight parameter $\gamma$, although most of our results are valid in arbitrary dimension. The regime $0 < \gamma < 1$ corresponds to the so-called \textit{hard potentials} while $\gamma < 0$ corresponds to the \textit{soft potentials} with a further classification of $-2 \leq \gamma < 0$ as the moderately soft potentials and $-4 \leq \gamma< -2$ as the very soft potentials. The particular instances of $\gamma=0$ and $\gamma=-d$ are known as the Maxwellian and Coulomb cases respectively. 
	
	The purpose of this work is to propose a new perspective inspired from gradient flows for weak solutions to~\eqref{eq:lan}, which is in analogy with the relationship of the heat equation and the 2-Wasserstein metric, see \cite{jordan_variational_1998, ambrosio_gradient_2008-1}. Our main result is inspired by and extends Erbar's work~\cite{erbar_gradient_2016}. There, he establishes the gradient flow perspective for the closely related spatially homogeneous Boltzmann equation with bounded collision kernels ($\gamma=0$) which we perform in the case of Landau for $\gamma\in(-3,0]$ (c.f.~\Cref{thm:fullequiv}). One of the fundamental steps is to symmetrize the right hand of \eqref{eq:lan}. More specifically, if we consider a test function $\phi \in C_c^\infty(\mathbb{R}^d)$ we can formally characterize the equation by
	\begin{equation}
		\label{eq:wklan}
		\frac{d}{dt}\int_{\mathbb{R}^d} \phi f dv = - \frac{1}{2}\iint_{\mathbb{R}^{2d}}  ff_* |v-v_*|^{2+\gamma} (\nabla \phi - \nabla_* \phi_*) \cdot \Pi[v-v_*](\nabla \log f - \nabla_*\log f_*)dv_* dv,
	\end{equation}
	where the change of variables $v\leftrightarrow v_*$ has been exploited. Building our analogy with the heat equation and the 2-Wasserstein distance, we define an appropriate gradient
	$$
	\tn\phi:= |v-v_*|^{1+\frac{\gamma}{2}}\Pi[v-v_*](\nabla \phi - \nabla_* \phi_*),
	$$
	so that equation~\eqref{eq:wklan} now looks like
	\[
	\frac{d}{dt}\iR \phi f dv = -\frac{1}{2} \iiR ff_* \tn \phi \cdot \tn \log f dv_* dv,
	\]
	noting that $\Pi^2 = \Pi$. To highlight the use of this interpretation, we notice that $\tn \phi=0$, when we choose as test functions $\phi = 1,\; v_i,\; |v|^2$ for $i=1,\dots,d$ which immediately shows that formally the equation conserves mass, momentum and energy. The action functional defining the Landau metric mimics the Benamou-Brenier formula \cite{benamou_computational_2000} for the 2-Wasserstein distance, see \cite{dolbeault_new_2009,E14,EM14} for other distances defined analogously for nonlinear and non-local mobilities. In fact, the Landau metric is built by considering a minimizing action principle over curves that are solutions to the appropriate continuity equation, that is
	\begin{equation}\label{dlmetric}
		d_L(f,g):=\min_{\substack{\partial_t\mu+\frac{1}{2}\tn\cdot(V\mu\mu_*)=0\\ \mu_0=f,\;\mu_1=g}} \left\{\frac{1}{2}\int_0^1 \iiR |V|^2\;d\mu(v)d\mu(v_*)dt\right\},
	\end{equation}
	where the $\tn\cdot$ is the appropriate divergence; the formal adjoint to the appropriate gradient  (see Section~\ref{sec:notation}).

	Also, we notice that analogously to the heat equation, written  as the continuity equation $\partial_t f= \nabla\cdot(f\nabla \log f)$, the Landau equation can be formally re-written as
	$$
	\partial_t f=\frac{1}{2}\tn\cdot(ff_*\tn \log f ),
	$$
	equivalent to the continuity equation with non-local velocity field given by
	\begin{equation}\label{eq:gradflow}
		\begin{cases}
			\partial_t f + \nabla\cdot( U(f) f)=0\\[2mm]
			\displaystyle U(f) := - \int_{\mathbb{R}^d} |v-v_*|^{2+\gamma}\Pi[v-v_*]\left(\nabla \log f - \nabla_*\log f_*\right)f_* dv_* \,.
		\end{cases} 
	\end{equation}
	This is a direct way to write~\eqref{eq:lan} in the form of a continuity equation. Considering the evolution of Boltzmann entropy we formally obtain
	\begin{align}
		\begin{split}
			\label{eq:entdiss}
			\frac{d}{dt}\int_{\mathbb{R}^d}f\log f dv 	&=: -D(f_t)=
			-\frac{1}{2}\iint_{\mathbb{R}^{2d}}|\tn \log f|^2 ff_* dv_* dv \leq 0.
		\end{split}
	\end{align} 
	In physical terms this is referred to as the \textit{entropy dissipation} referred to as entropy production in the physics literature from defining $\mathcal{H}$ with a minus sign) since it formally shows that the entropy functional 
	$$
	\mathcal{H}[f] := \int_{\mathbb{R}^d} f \log f dv
	$$
	is non-increasing along the dynamics of the Landau equation. Moreover, by integrating equation~\eqref{eq:entdiss} in time one formally obtains
	\begin{align}
		\label{eq:entprod}
		\begin{split}
			\mathcal{H}[f_t]	&+ \int_0^t D(f_s)ds = \mathcal{H}[f_0].
		\end{split}
	\end{align}
	In \cite{villani_new_1998}, Villani introduced the notion of H-solution, which captures this formal property. Motivated by the physical considerations of certain conserved quantities and entropy dissipation,  H-solutions provided a step towards well-posedness of the Landau equation in the soft potential case. One advantage to this approach is that it avoids assuming that the solutions belongs to $L^p(\mathbb{R}^3)$ for $p>1$. For moderately soft potentials, the propagation of $L^p$ norms is proven and this is enough to make sense of classical weak solutions~\cite{wu_global_2013}. In the very soft potential case, there is no longer a guarantee of $L^p$ propagation due to the singularity of the weight. We refer to~\cite[Section 1.2]{desvillettes_entropy_2015} for a heuristic description of this difficulty.
	
	Similar to H-solutions our approach will also be based on the entropy dissipation \eqref{eq:entprod}. Following De Giorgi's minimizing movement ideas \cite{ambrosio_min,ambrosio_gradient_2008-1}, we characterize the Landau equation by its associated Energy Dissipation Inequality. More specifically, we show that weak solutions to \eqref{eq:lan} with initial data $f_0$ are completely determined by the following functional inequality: 
	$$
	\mathcal{H}[f_t]+\frac{1}{2}\int_0^t |\dot{f}|^2_{d_L}(s)\;ds+\frac{1}{2}\int_0^t D(f_s)\;ds\le \mathcal{H}[f_0]\qquad\mbox{for a.e. every $t>0$,}
	$$
	where $|\dot{f}|^2_{d_L}(s)$ stands for the metric derivative associated to the Landau metric defined above. Our analysis is also largely inspired by Erbar's approach in viewing the Boltzmann equation as a gradient flow~\cite{erbar_gradient_2016} and recent numerical simulations of the homogeneous Landau equation in~\cite{carrillo_particle_2019} based on a regularized version of \eqref{eq:gradflow}. In contrast with the classical 2-Wasserstein metric, one of the main features of the Landau equation~\eqref{eq:lan} and metric~\eqref{dlmetric} is that they are non-local. To be precise, gradient flow theory has been successfully applied to the study of many non-local PDEs~\cite{carrillo_confinement_2012,BCC08,CLSS10} by viewing them as gradient flows of appropriate energy functionals with respect to the 2-Wasserstein metric. The novelty in this work is the construction of the \textit{non-local} metric $d_L$ with respect to which~\eqref{eq:lan} can be viewed as the gradient flow of $\mathcal{H}$. Hence, the convergence analysis usually relying on convexity and lower-semi continuity needs to be adapted to deal with the non-locality of this equation. In particular, our characterization Theorem~\ref{thm:fullequiv} is based in using (expected) a-priori estimates to deal with the non-locality through appropriate bounds.
	
	On the other hand, the state of the art related to the uniqueness for the Landau equation depends on the range of values $\gamma$ may take. In the cases of hard potentials or Maxwellian, the uniqueness theory is very well understood due to Villani and the third author~\cite{desvillettes_spatially_2000-1,desvillettes_spatially_2000,villani_spatially_1998}. In the soft potential case, one of the first major contributions to the general theory of the spatially inhomogeneous Landau equation ($\gamma \geq -3$) was the global existence and uniqueness result by Guo~\cite{guo_landau_2002}. This result was achieved in a perturbative framework with high regularity assumptions on the initial data. Through probabilistic arguments, the next major improvement to uniqueness for $\gamma \in (-3,0)$ came from Fournier and Gu\'erin~\cite{fournier_well-posedness_2009}. Their result established uniqueness in a class of solutions that shrinks as $\gamma$ decreases towards $-3$, as more $L^p$ and moments assumptions are needed. In their proof, uniqueness is shown by proving stability with respect to the 2-Wasserstein metric.
	
	Still lots of open questions for the soft potential case remain. In particular, a fundamental question like uniqueness for the Coulomb case is unresolved. To tackle this and other problems an array of novel methods have been employed. Here is an incomplete sample of the contributions made in this direction which highlight the difficulties of the soft potential case~\cite{desvillettes_spatially_2000,desvillettes_spatially_2000-1,alexandre_priori_2015,carrapatoso_landau_2015,carrapatoso_estimates_2015,wu_global_2013,gualdani_spectral_2016,gualdani2016estimates,gualdani_global_2017,gualdani_review_2018,strain_uniqueness_2019, GIMV19, S17,GGIV19}. A brief glance at some of these references illustrates the breadth of techniques that have found partial success at answering the open questions; probability-based arguments, kinetic and parabolic theory, and many more.
	
	The purpose of this paper is to bring in another set of techniques to help answer some of these fundamental questions. The gradient flow theory applied to PDEs has flourished in the last decades. In their seminal paper~\cite{jordan_variational_1998}, Jordan, Kinderlehrer, and Otto proposed a variational approach (JKO scheme) extended later on to a wide class of PDEs using the optimal transportation distance of probability measures. These results and many more achievements from their contemporaries allowed for novel approaches to questions of existence, uniqueness, convergence to equilibrium, and other aspects of a large class of PDE; we mention~\cite{ambrosio_gradient_2008-1,santambrogio_euclidean_2017} for a coherent exposition of these techniques and the relevant literature, even as more advances have been made since then. 
	
	The advantage of our variational characterization of the Landau equation is that it unveils new possible routes of showing convergence results for this equation. First of all, it allows for natural regularizations of the Landau equation by taking the steepest descent of regularized entropy functionals instead of the Boltzmann entropy as in \cite{MR3913840}. This idea was recently developed in \cite{carrillo_particle_2019} leading to structure preserving particle schemes with good accuracy. We can also consider the framework of convergence of gradient flows based on $\Gamma$-convergence introduced in \cite{SS04,serfaty2011} to attack the convergence of these numerical methods \cite{carrillo_particle_2019}. Moreover, this approach is flexible enough to also study the rigorous convergence of the grazing collision limit of the Boltzmann equation to the Landau equation. The grazing collision limit was recently revisited in the gradient flow framework by three of the authors~\cite{CDW21}. There, ideas from $\Gamma$-convergence were used to pass from Erbar's gradient flow description for the Boltzmann equation~\cite{erbar_gradient_2016} to the present work's description of the Landau equation.
	Finally, deriving uniqueness from the variational structure is classically done through convexity properties of the entropy functional with respect to the geodesics of the Landau metric. This is another important avenue of research that our work opens. Moreover, gradient flows of convex entropies typically enjoy instantaneous smoothing~\cite{ambrosio_gradient_2008-1}; even if the entropy at $t=0$ is infinite, for $t>0$, the entropy becomes finite. In the case of Landau, we are not aware if this property holds for $\mathcal{H}$.
	
	We mention briefly the connection between~\eqref{eq:lan} and the Fokker-Planck equation. For $\gamma=0$, one can formally compute the evolution of $\int v^i v^jf(v)dv$ through~\eqref{eq:lan}. This a priori information allows one to reduce~\eqref{eq:lan} to a linear Fokker-Planck equation for $\gamma=0$.
The present work proposes the alternative viewpoint that the resultant Fokker-Planck equation can be viewed as the $d_L$-gradient flow of $\mathcal{H}$ for $\gamma=0$. Since many variants of the linear Fokker-Planck equation have been well-studied, this case serves as a nice benchmark to test the gradient flow theory developed here.
	
	The plan of this paper is as follows. Section~\ref{sec:prelim} introduces the prerequisites and contains the  statements of the main results. We first construct and analyze in Section~\ref{sec:Landist} the Landau metric based on \eqref{dlmetric}. For a regularized problem, Section~\ref{sec:EDE} shows the equivalence between weak solutions and gradient flows, while Section~\ref{sec:JKO} shows the existence of gradient flow solutions via a Minimizing Movement scheme. Finally, we show in Section 6 that a gradient flow solution is equivalent to H-solutions of the Landau equation \eqref{eq:lan} under some integrability assumptions. Appendix A is devoted to some technical lemmata needed in the proof of the main theorems regarding the chain rule identity behind the definition of weak solutions for the regularized Landau equation.

	
	\section{Preliminaries and the main results}
	\label{sec:prelim}
	We start by introducing the necessary notation and definitions together with a quick overview of gradient flow concepts to make our main results fully self-contained.
	
	\subsection{Notations and definitions}
	\label{sec:notation}
	We denote
	\[
	a \lesssim_{\dots} b \iff \exists C(\dots)>0 \text{ s.t. } a \leq C(\dots)b.
	\]
	We adopt the Japanese angle bracket notation for a smooth alternative to absolute value
	\[
	\japangle{v}^2 = 1+ |v|^2, \quad v \in \Rd.
	\]
	For $\epsilon>0$, we denote our regularization kernel to be an exponential distribution
	\[
	G^{\epsilon}(v) = \epsilon^{-d} G\left(
	v/\epsilon
	\right), \quad G(v) = C_d \exp\left(
	-\japangle{v}
	\right), \quad C_d = \left(
	\iR \exp (-\japangle{v}) dv
	\right)^{-1}.
	\]
	Our results work for some general tail behaviour in the kernels given by
	\[
	G^{s,\epsilon}(v) = \epsilon^{-d}G^s(v/\epsilon), \quad G^s(v) = C_{s,d} \exp(-\japangle{v}^s), \quad C_{s,d} = \left(\iR \exp(-\japangle{v}^s)dv\right)^{-1},
	\]
	for $s>0$; we point out some of the limitations and restrictions on $s>0$ in the later estimates. We shall refer to $G^{2,\epsilon}$ as the Maxwellian regularization. We denote the space of probability measures over $\mathbb{R}^d$ by $\mathscr{P}(\Rd)$, endowed with the weak topology against bounded continuous functions. We will mostly be dealing with the Lebesgue measure on $\Rd$ as our reference measure which we denote by $\mathcal{L}$. The subset $\mathscr{P}^a(\Rd) \subset \mathscr{P}(\Rd)$ denotes the set of absolutely continuous probability measures with respect to Lebesgue measure. For $p>0$, we also define the probability measures with finite $p$-moments $\mathscr{P}_p(\mathbb{R}^d)$ by
	\[
	\mathscr{P}_p(\mathbb{R}^d):= \left\{
	\mu \in \mathscr{P}(\mathbb{R}^d) \, \bigg| \, m_p(\mu) :=  \iR  \japangle{v}^p d\mu(v) < \infty
	\right\}.
	\]
	Finally, for $E>0$, we consider the subset $\mathscr{P}_{p,E}(\Rd)\subset \mathscr{P}_p(\Rd)$ of probability measures with $p$-moments uniformly bounded by $E$;
	\[
	\mathscr{P}_{p,E}(\Rd) := \left\{
	\mu \in \mathscr{P}_p(\Rd) \, \bigg| \, m_p(\mu) \leq E
	\right\}.
	\]
	
	We denote by $\mathcal{M}$ the space of signed Radon measures on $\mathbb{R}^d\times \Rd$ with the standard weak* topology against the continuous and compactly supported functions of $\Rd\times \Rd$. The space $\mathcal{M}^d$ is the space of signed $d$-length Radon measures. For $T>0$, we will add the time contribution of the measures by denoting $\mathcal{M}_T$ to be the space of signed Radon measures on $\mathbb{R}^d\times \Rd \times [0,T]$ with the usual weak* topology. Similarly, $\mathcal{M}_T^d$ will be the space of signed $d$-length Radon measures on $\Rd \times \Rd \times [0,T]$.  
	
	For $\mu\in\mathscr{P}(\Rd)$, we define a family of regularized entropies $\mathcal{H}_\epsilon[\mu]$ by
	\[
	\mathcal{H}_\epsilon[\mu] := \iR [\mu*G^\epsilon](v)\log [\mu*G^\epsilon](v)dv,
	\]
	which we shall see is well-defined provided $\mu$ has a finite moment in Lemma~\ref{lem:carlen}. Formally, one can calculate the first variation of this functional in $\mathscr{P}_2$ as
	\[
	\firstvar{\mathcal{H}_\epsilon}{\mu}(v) = G^\epsilon * \log [\mu*G^\epsilon](v).
	\]
	This can be formally obtained by calculating Fr\'echet derivatives in the sense of identifying the following limit
\[
\int_{\R^d}\frac{\delta \mathcal{H}_\epsilon}{\delta \mu}(v) \phi(v)dv = \lim_{t\downarrow 0}\frac{\mathcal{H}_\epsilon[\mu + t\phi] - \mathcal{H}_\epsilon[\mu]}{t},
\]
for arbitrary $\phi\in C_c^\infty(\R^d)$ with zero mean $\int_{\R^d}\phi = 0$. To be precise, the first variation (in an $L^2$ setting) would actually be $\frac{\delta \mathcal{H}_\epsilon}{\delta \mu} = 1 + G^\epsilon * \log [\mu * G^\epsilon]$. We drop the constant term since our functional space is $\mathscr{P}$ and the first variation typically appears with derivatives applied to it. For a functional $\mathcal{F}: \mathscr{P}^a(\Rd) \to \R$ with first variation $\firstvar{\mathcal{F}}{f}$, we refer to the $\mathcal{F}$ Landau equation as
	\begin{equation}
		\label{eq:FLandau}
		\partial_tf = \nabla\cdot \left(
		f\iR f_* |v-v_*|^{2+\gamma} \Pi[v-v_*]\left(
		\nabla \firstvar{\mathcal{F}}{f}- \nabla_* \firstvar{\mathcal{F}_*}{f_*}
		\right)dv_*
		\right).
	\end{equation}
	To clarify the meaning of $\tn\cdot$, for a given test function $\phi = \phi(v) \in \R^d$ and vector-valued test function $A = A(v,v_*) \in \R^d$, we have
	\[
	\iiR [\tn\phi](v,v_*)\cdot A(v,v_*) dv_* dv = -\iR \phi(v) [\tn\cdot A](v) dv.
	\]
	In this way, the $\mathcal{F}$ Landau equation~\eqref{eq:FLandau} can be concisely written as
	\[
	\partial_t f = \frac{1}{2}\tn \cdot \left(
	ff_* \tn\firstvar{\mathcal{F}}{f}
	\right).
	\]
	Note, by formally testing~\eqref{eq:FLandau} with $\phi = \firstvar{\mathcal{F}}{f}$, one obtains an analogy of Boltzmann's H-theorem with the functional $\mathcal{F}$;
	\[
	\frac{d}{dt}\mathcal{F}[f_t] = -D_{\mathcal{F}}(f_t) := -\frac{1}{2}\iiR ff_* \left|
	\tn \firstvar{\mathcal{F}}{f}
	\right|^2 dvdv_*\leq 0.
	\]
	We will refer to $D_\mathcal{F}$ as the $\mathcal{F}$ dissipation. These notations induce our notion of weak solutions to the $\mathcal{F}$ Landau equation~\eqref{eq:FLandau} closely following Villani's H-solutions~\cite{villani_new_1998}.
	
	\begin{definition}[Weak $\mathcal{F}$ solutions]
		\label{def:wkwFLandau}
		For $T>0$, we say that a curve $f\in C([0,T];L^1(\Rd))$ is a weak solution to the $\mathcal{F}$ Landau equation~\eqref{eq:FLandau} if the following hold.
		\begin{enumerate}
			\item $f\mathcal{L}$ is a probability measure with uniformly bounded second moment so that
			\[
			f_t\geq 0, \quad \iR f_t(v) dv = 1, \quad \forall t\in[0,T], \quad \sup_{t\in[0,T]} \iR \japangle{v}^2 f_t(v) dv < \infty.
			\]
			\item The functional $\mathcal{F}$ evaluated along the curve is bounded by its initial value
				\[
				\mathcal{F}[f_t] \le \mathcal{F}[f_0] <+\infty, \quad \forall t\in[0,T].
				\]
			\item The $\mathcal{F}$ dissipation is time integrable
			\[
			\int_0^T D_{\mathcal{F}}(f_t)dt = \frac{1}{2}\int_0^T \iiR ff_* \left|
			\tn \firstvar{\mathcal{F}}{f}
			\right|^2 dvdv_* dt < \infty.
			\]
			\item For every test function $\phi\in C_c^\infty((0,T)\times \Rd)$, equation~\eqref{eq:FLandau} is satisfied in weak form
			\[
			\int_0^T\iR \partial_t \phi f_t(v) dv dt = \frac{1}{2} \int_0^T \iiR ff_* \tn \phi \cdot \tn \firstvar{\mathcal{F}}{f}dvdv_*dt.
			\]
		\end{enumerate}
	\end{definition}
	For $\epsilon>0$, we will refer to the weak $\mathcal{H}_\epsilon$ solutions as \textit{$\epsilon$-solutions} and, recalling $\mathcal{H}$ is the Boltzmann entropy, we will refer to weak $\mathcal{H}$ solutions as just \textit{weak solutions} or \textit{H-solutions}. We deliberately use the terminology of H-solutions since the time integrability of $D_{\mathcal{H}}(f_t)$, as for Villani~\cite{villani_new_1998}, is essential in our analysis.
	\subsection{Quick review of gradient flow theory}
	\label{sec:revgradflow}
	We recall the basic definitions of gradient flow theory that can be found in more generality in~\cite[Chapter 1]{ambrosio_gradient_2008-1}. Throughout, $(X,d)$ denotes a complete (pseudo)-metric space $X$ with (pseudo)-metric $d$. Points $a<b\in\mathbb{R}$ will refer to endpoints of some interval. $F:X\to (-\infty,\infty]$ will denote a proper function.
	\begin{definition}[Absolutely continuous curve]
		\label{defn:abscontcurve}
		A function $\mu : t\in(a,b) \mapsto \mu_t\in X$ is said to be an \textit{absolutely continuous curve} if there exists $m\in L^2(a,b)$ such that for every $s\leq t\in(a,b)$
		\[
		d(\mu_t,\mu_s) \leq \int_s^tm(r)dr.
		\]
	\end{definition}
	Among all possible functions $m$ in Definition~\ref{defn:abscontcurve}, one can make the following minimal selection.
	\begin{definition}[Metric derivative]\label{defn:metder}
		For an absolutely continuous curve $\mu : (a,b) \to X$, we define its \textit{metric derivative} at every $t\in(a,b)$ by
		\[
		|\dot{\mu}|(t) := \lim_{h\to0}\frac{d(\mu_{t+h},\mu_t)}{|h|}.
		\]
	\end{definition} 
	Further properties of the metric derivative can be found in~\cite[Theorem 1.1.2]{ambrosio_gradient_2008-1}.
	
	\begin{definition}[Strong upper gradient]
		\label{defn:sug}
		The function $g:X\to [0,\infty]$ is a \textit{strong upper gradient} with respect to $F$ if for every absolutely continuous curve $\mu : t\in (a,b) \mapsto \mu_t \in X$ we have that $g \circ \mu :(a,b) \to [0,\infty]$ is Borel and the following inequality holds
		\[
		|F[\mu_t] - F[\mu_s]| \leq \int_s^t g(\mu_r)|\dot{\mu}|(r)dr, \quad \forall a < s \leq t <b.
		\]
	\end{definition}
	Using Young's inequality and moving everything to one side, the inequality in Definition~\ref{defn:sug} implies
	\[
	F[\mu_t] - F[\mu_s] +\frac{1}{2} \int_s^t g(\mu_r)^2 dr + \frac{1}{2}\int_s^t |\dot{\mu}|^2(r) dr \geq 0, \quad \forall a< s \leq t < b.
	\]
	If the reverse inequality also holds, one obtains the stronger Energy Dissipation \textit{Equality}. This leads to our notion of gradient flows.
	\begin{definition}[Curve of maximal slope]
		\label{defn:cms}
		An absolutely continuous curve $\mu :(a,b)\to X$ is said to be a \textit{curve of maximal slope} for $F$ with respect to its strong upper gradient $g: X \to [0,\infty]$ if $F \circ \mu : (a,b) \to [0,\infty]$ is non-increasing and the following inequality holds
		\[
		F[\mu_t] - F[\mu_s] +\frac{1}{2} \int_s^t g(\mu_r)^2 dr + \frac{1}{2}\int_s^t |\dot{\mu}|^2(r) dr \leq 0, \quad \forall a< s \leq t < b.
		\]
	\end{definition}
	$F$ has the following natural candidates for upper gradient.
	\begin{definition}[Slopes]
		\label{defn:slope}
		We define the \textit{local slope of $F$} by
		\[
		|\partial F|(\mu) := \limsup_{\nu\to \mu}\frac{(F(\nu) - F(\mu))^+}{d(\nu,\mu)}.
		\]
		The superscript `+' refers to the positive part. The \textit{relaxed slope of $F$} is given by
		\[
		|\partial^-F|(\mu) := \inf\{
		\liminf_{n\to\infty}|\partial F|(\mu_n) \, : \, \mu_n \to \mu, \, \sup_{n\in\mathbb{N}}(d(\mu_n,\mu),F(\mu_n)) < +\infty
		\}.
		\]
	\end{definition}

	
	\subsection{Main results}
	\label{sec:mainresults}
	In order to understand the Landau equation as a gradient flow, we need to clarify what type of object the corresponding metric is. 
	\begin{theorem}[Distance on $\mathscr{P}_{2,E}(\Rd)$]
		\label{them:existdL}
		The (pseudo)-metric $d_L$ on $\mathscr{P}_{2,E}(\R^d)$ satisfies:
		\begin{itemize}
			\item $d_L$-convergent sequences are weakly convergent.
			\item   $d_L$-bounded sets are weakly compact.
			\item   The map $(\mu_0,\mu_1)\mapsto d_L(\mu_0,\mu_1)$ is weakly lower semicontinuous.
			\item For any $\tau\in\mathscr{P}_2(\R^d)$ the subset $\mathscr{P}_\tau(\Rd):= \left\{ \mu \in \mathscr{P}_{2,m_2(\tau)}(\R^d) \, | \, d_L(\mu,\tau)<\infty \right\}$ is a complete geodesic space.
		\end{itemize}
	\end{theorem}
	The content of this theorem is essentially that our new proposed distance actually provides a meaningful topological structure on $\mathscr{P}_{2,E}(\R^d)$. Furthermore, the connection to $\epsilon$-solutions of Landau is established when considering the previous notions of slope and upper gradient with respect to $d_L$. General conditions which guarantee $d_L(\mu_0,\mu_1)<+\infty$ are presently unknown. In~\Cref{lem:conserve}, we will see that a necessary condition is that $\mu_0$ and $\mu_1$ have the same mean velocity. Moreover, for $\gamma\in[-4,-2]$, \Cref{lem:conserve} asserts that they should have the same second moment. In the construction of $d_L$ detailed in~\Cref{sec:Landist}, if $\mu = \mu(t)$ for $t\in[0,T]$ is an H-solution of Landau, then it is certainly true that $d_L(\mu(t), \, \mu(s))<+\infty$ for all $0 \le t, s \le T$.
	\begin{theorem}[Epsilon equivalence]
		\label{thm:epsequiv}
		Fix any $\epsilon, E>0, \gamma\in [-4,0]$. Assume that a curve $\mu : [0,T] \to \PtwoE$ has a density $\mu_t = f_t \mathcal{L}$. Then $\mu$ is a curve of maximal slope for $\mathcal{H}_\epsilon$ with respect to its upper gradient $\sqrt{D_{\mathcal{H}_\epsilon}}$ if and only if its density $f$ is an $\epsilon$-solution to the Landau equation. 
	\end{theorem}
	From the numerical perspective, we can also construct $\epsilon$-solutions using the JKO scheme (see Section~\ref{sec:JKO}) which is the following
	\begin{theorem}[Existence of curves of maximal slope]
		\label{thm:existcms}
		For any $\epsilon,E>0, \gamma \in [-4,0]$, and initial data $\mu_0 \in \mathscr{P}_{2,E}(\Rd)$, there exists a curve of maximal slope in $\mathscr{P}_{2,E}(\Rd)$ for $\mathcal{H}_\epsilon$ with respect to its upper gradient $\sqrt{D_{\mathcal{H}_\epsilon}}$.
	\end{theorem}
\begin{remark}
The curves constructed in~\Cref{thm:existcms} do not necessarily have a density with respect to Lebesgue measure; the regularization allows $\mathcal{H}_\epsilon[\mu]<+\infty$ without $\mu$ being absolutely continuous with respect to Lebesgue measure. Moreover, uniqueness of such curves is beyond the scope of the present work although it would be interesting to see what convexity properties are available for $\mathcal{H}_\epsilon$ with respect to $d_L$. This could also shed some insight into the available convexity of $\mathcal{H}$ with respect to $d_L$.
\end{remark}
	\begin{remark}
		The choice of an exponential convolution kernel $G^\epsilon$ for the regularized entropy $\mathcal{H}_\epsilon$ is perhaps unnatural compared to the Maxwellian regularization $G^{2,\epsilon}$. We discuss in more detail the estimates that fail using $G^{2,\epsilon}$ in Remark~\ref{rk:Glogdiff} as it pertains to Theorem~\ref{thm:epsequiv}. With respect to Theorem~\ref{thm:existcms}, the general construction of \textit{some} curve can be done even with the Maxwellian regularization. However, due to the same lack of estimates, this curve might not be a curve of maximal slope with respect to $\sqrt{D_{\mathcal{H}_\epsilon}}$. This is discussed in Remark~\ref{rk:JKOremarks}.
	\end{remark}
	Motivated by recent numerical experiments~\cite{carrillo_particle_2019}, Theorems~\ref{thm:epsequiv} and~\ref{thm:existcms} provide the theoretical basis to this $\epsilon$ approximated Landau equation. In the limit $\epsilon\to0$, more assumptions are required.
	\begin{theorem}[Full equivalence]
		\label{thm:fullequiv}
		We fix $d=3$ and $\gamma\in (-3,0]$. Suppose that for some $T>0$, a curve $\mu : [0,T] \to \mathscr{P}(\R^3)$ has a density $\mu_t = f_t \mathcal{L}$ that satisfies the following set of assumptions
		\begin{enumerate}[label=(\textbf{A\arabic*})]
			\item \label{itm:ass1}(Moments and $L^p$) Assume that for some $0 < \eta \leq \gamma + 3$, we have
			\[
			\japangle{v}^{2-\gamma} f_t(v) \in L_t^\infty(0,T;L_v^1\cap L_v^\frac{3-\eta}{3+\gamma-\eta}(\R^3)).
			\]
			\item \label{itm:ass2}(Finite entropy) We assume that the initial entropy is finite
		\[
		\mathcal{H}[f_0] = \int_{\R^3} f_0 \log f_0 < +\infty.
		\]	
			\item \label{itm:ass3}(Finite entropy-dissipation) We assume that the entropy-dissipation of $f$ is integrable in time
			\begin{align*}
				&D(f_t) = D_{\mathcal{H}}(f_t) = \frac{1}{2} \iiRsix ff_* \left|
				\tn 
				\firstvar{\mathcal{H}}{f}
				\right|^2 dvdv_* =  	\\
				&\frac{1}{2} \iiRsix ff_* |v-v_*|^{\gamma+2} |\Pi[v-v_*](\nabla \log f - \nabla_*\log f_*)|^2 dvdv_* \in L_t^1(0,T).
			\end{align*} 
		\end{enumerate}
		Then $\mu$ is a curve of maximal slope for $\mathcal{H}$ with respect to its upper gradient $\sqrt{D}$ if and only if its density $f$ is a weak solution of the Landau equation.
	\end{theorem}
	
	\begin{remark}
		When $\gamma \in [-2,0]$, it is known that for suitable initial data
		(lying in weighted $L^p$ spaces for $p$ large enough and for a
		sufficient power-like weight), weak solutions of Landau equation
		satisfying~\ref{itm:ass1}--\ref{itm:ass3} are known to exist (and to be strong and
		unique under extra conditions). We refer to \cite{wu_global_2013},
		and Appendix B of \cite{LD_PX} when $\gamma > -2$, for
		details.
		
		When $\gamma \in (-3, -2)$, Assumption~\ref{itm:ass1} is not known to hold for
		global weak solutions with large initial data. Solutions satisfying~\ref{itm:ass1}--\ref{itm:ass3} are nevertheless known to exist for initial data close to
		equilibrium (cf. \cite{guo_landau_2002}, in a much larger spatially
		inhomogeneous context), or in the Coulomb case $\gamma = -3$ (in that case $\frac{3 - \eta}{3 + \gamma - \eta}$ being replaced by $\infty$) for large initial data, but on specific intervals of times only (\cite{DHJ20,AP77}).
	
	The focus on the Maxwellian and soft potential regime $\gamma\le 0$ here is motivated by building a gradient flow framework to address the open questions for Landau. The hard potential case $\gamma\in(0,1)$ has already been studied in detail by the third author and Villani~\cite{desvillettes_spatially_2000,desvillettes_spatially_2000-1}. We believe that our results also carry to the hard potentials. In particular, the exponents in assumption~\ref{itm:ass1} should be modified to
	\[
	\japangle{v}^{2+\gamma}f_t(v) \in L_t^\infty(0,T; \, L_v^1(\R^3)), \quad f_t(v)\in L_t^\infty(0,T;\, L_v^{\frac{3}{3-\gamma}+}(\R^3)), \quad 0 < \gamma< 1.
	\]
	We emphasize that these conditions are guaranteed since the required moments and $L^p$ integrability are propagated from appropriate initial data when $\gamma> 0$~\cite{desvillettes_spatially_2000,desvillettes_spatially_2000-1}.	This condition appears in~\cite[Corollary 2.7]{desvillettes_entropy_2016}. It is the hard potential version of~\Cref{thm:desvillettes16} which is crucial to the proof of~\Cref{thm:fullequiv}. Much of our analysis remains the same, however the space $\mathscr{P}_2$ should be changed to $\mathscr{P}_{2+\gamma}$ cohering with the moment condition above and trivializing~\Cref{lem:Jf}, for example.
\end{remark}
	
	It is an open problem to find the range of values $\gamma$ under which we can show the existence of curves of maximal slope for the original Landau equation
	\eqref{eq:lan}, or equivalently, contructing solutions of the original Landau equation passing $\epsilon \to 0$ in Theorem \ref{thm:existcms}. Some of the difficulties to achieve this result are the propagation of moments for the regularized Landau equation uniformly in $\epsilon$ and the compactness of sequences with bounded in $\epsilon$ regularized entropy dissipation $D_{\mathcal{H}_\epsilon}$. The rest of this work is devoted to show the main four theorems in the next four sections.
	
	\section{The Landau metric $d_L$}
	\label{sec:Landist}
	Our approach to defining the distance $d_L$ mentioned in Theorem~\ref{them:existdL} closely follows the dynamic formulation of transport distances originally due to Benamou and Brenier~\cite{benamou_computational_2000} and further extended by Dolbeault, Nazaret, and Savar\'e~\cite{dolbeault_new_2009}. We also refer the reader to Erbar~\cite{erbar_gradient_2016} for a similar approach.
	\subsection{Grazing continuity equation}
	\label{sec:GCE}
	We consider for $\gamma\in[-4,0]$ the \textit{grazing continuity equation}:
	\begin{equation}\label{eq:GCE}
		\partial_t \mu_t + \frac{1}{2}\tilde{\nabla}\cdot M_t = 0, \quad \text{ in } (0,T) \times\mathbb{R}^{d},
	\end{equation}
	which is interpreted in the sense of distributions. For every $\phi \in C_c^\infty ((0,T)\times \Rd)$, we have
	\begin{equation*}
		\int_0^T \int_{\mathbb{R}^d} \partial_t \phi(t,v) d\mu_t(v)  dt + \frac{1}{2}\int_0^T \iint_{\mathbb{R}^{2d}} [\tilde{\nabla}\phi](t,v,v_*) dM_t(v,v_*) dt = 0.
	\end{equation*}
	Another formulation(see~\Cref{lem:ctsrep}) is the following for $\zeta\in C_c^\infty(\Rd)$, 
	\begin{equation}
		\label{eq:GCEtdiff}
		\frac{d}{dt}\iR \zeta(v) d\mu_t(v) = \frac{1}{2}\iiR \tn \zeta(v,v_*)dM_t(v,v_*).
	\end{equation}
	The curves $(\mu_t)_{t\in[0,T]}, (M_t)_{t\in[0,T]}$ are Borel families of measures belonging to $\mathcal{M}_+$ and $\mathcal{M}^d$ respectively. We will refer to $\mu$ from the pair as a \textit{curve} and $M$ as a \textit{grazing rate}. For some regularity properties, we will also need to assume the following moment condition
	\begin{equation}
		\label{eq:finitemuM}
		\int_0^T \iiR (1+|v|+|v_*|)d|M_t|(v,v_*)dt < \infty.
	\end{equation} 
	We first establish some a-priori properties of solutions to the grazing continuity equation.
	\begin{lemma}[Continuous representative]\label{lem:ctsrep}
		For families $(\mu_t),(M_t)$ satisfying the grazing continuity equation and the finite moment condition \eqref{eq:finitemuM}, there exists a unique weakly* continuous representative curve $(\tilde{\mu}_t)_{t\in[0,T]}$ such that $\tilde{\mu}_t = \mu_t$ a.e. $t\in[0,T]$. Furthermore, for any $\phi \in C_c^\infty ((0,T) \times \mathbb{R}^d)$ and any $t_0,t_1 \in [0,T]$, we have the following formula
		\begin{equation*}
			\iR \phi_{t_1}d\tilde{\mu}_{t_1} - \iR \phi_{t_0}d\tilde{\mu}_{t_0} = \int_{t_0}^{t_1}\iR \partial_t \phi d\mu_tdt + \frac{1}{2}\int_{t_0}^{t_1}\iiR \tilde{\nabla}\phi dM_tdt.
		\end{equation*}
	\end{lemma}
	\begin{proof}
		This proof is nearly identical to~\cite[Lemma 8.1.2]{ambrosio_gradient_2008-1}. There, it was crucial to estimate the distributional time derivative of $t\mapsto \mu_t$. We perform the analogous estimate here to highlight the difference in our context. Fix $\zeta\in C_c^\infty(\Rd)$ and consider the map
		\[
		t\in (0,T) \mapsto \mu_t(\zeta) = \iR \zeta(v) d\mu_t(v) \in \R.
		\]
		According to~\eqref{eq:GCEtdiff}, the distributional time derivative is
		\begin{align*}
			\dot{\mu}_t(\zeta) = \frac{1}{2}\iiR \tn \zeta dM_t(v,v_*) = \frac{1}{2}\iiR |v-v_*|^{1+\frac{\gamma}{2}}\Pi[v-v_*](\nabla\zeta - \nabla_* \zeta_*)dM_t(v,v_*).
		\end{align*}
	Depending on the values of $\gamma$ above or below -2, the integrand can be estimated
\begin{align*}
&\quad \left||v-v_*|^{1+\frac{\gamma}{2}}\Pi[v-v_*](\nabla\zeta - \nabla_* \zeta_*)\right| 	\\
&\le \left\{
\begin{array}{cl}
2^{1+\frac{\gamma}{2}}\sup_{w\in\Rd} |\nabla \zeta(w)| (|v|^{1+\frac{\gamma}{2}}+|v_*|^{1+\frac{\gamma}{2}}), 	&\gamma\in[-2,0] 	\\
\sup_{w\in\Rd} |D^2\zeta(w)| |v-v_*|^{2+\frac{\gamma}{2}}, 	&\gamma\in[-4,-2)
\end{array}
\right..
\end{align*}	
Consequently, using the moment condition~\eqref{eq:finitemuM}, we have the following estimates depending on $\gamma\in[-4,0]$,
		\begin{align*}
			|\dot{\mu}_t(\zeta)| &\lesssim \left\{
			\begin{array}{cl}
				\sup_{w\in\Rd}|\nabla\zeta(w)|\iiR  (1+|v|+|v_*|)d|M_t|(v,v_*),    &\gamma\in[-2,0]  \\
				\sup_{w\in\Rd}|D^2\zeta(w)|\iiR (1+|v|+|v_*|)d|M_t|(v,v_*),      &\gamma\in[-4,-2) 
			\end{array}
			\right..
		\end{align*}
		The rest of the proof proceeds as in~\cite[Lemma 8.1.2]{ambrosio_gradient_2008-1} using the $C^2$-norm of $\zeta$ for the soft potentials $\gamma\in[-4,-2)$ as opposed to their $C^1$ control of $\zeta$.
	\end{proof}
	\begin{lemma}[Conservation lemma]\label{lem:conserve}
		Fix $\gamma\in[-4,0]$ and let $(\mu_t)_{t\in[0,T]}, (M_t)_{t\in[0,T]}$ be Borel families of measures in $\mathcal{M}_+,\;\mathcal{M}^d$ respectively satisfying~\eqref{eq:GCE} and the moment condition~\eqref{eq:finitemuM}. Assume further that $(\mu_t)_{t\in[0,T]}$ is weakly* continuous with respect to $t$. We have that mass and momentum are conserved;
		\[
		\mu_t(\mathbb{R}^d) = \mu_0(\mathbb{R}^d), \quad \int_{\R^d}v \,d\mu_t(v) = \int_{\R^d}v\, d\mu_0(v), \quad \forall t \in [0,T].
		\]
		In the case $\gamma\in[-4,-2]$ we have that the energy is conserved;
		\[
		\int_{\R^d}|v|^2 \, d\mu_t(v) = \int_{\R^d}|v|^2 \, d\mu_0(v), \quad \forall t\in[0,T].
		\]
	\end{lemma}
	\begin{proof}
		To minimize clutter, we introduce $w = |v-v_*|^{1+\frac{\gamma}{2}}$. We show the proof of the conservation of energy for $\gamma\in[-4,-2]$. We consider a fixed $\varphi\in C^\infty_c(B_2)$ which satisfies
		$$
		0\le \varphi\le 1\qquad\mbox{and}\qquad \varphi(v)=1\qquad\mbox{in $B_1$}.
		$$
		We denote
		$$
		\varphi_R(v)=\varphi(v/R).
		$$
		Using the grazing continuity equation, we have that
		\begin{equation}\label{eq:Renergy}
			\begin{array}{l}
				\displaystyle\int_{\R^d} |v|^2\varphi_R(v)\;d\mu_t(v)-\int_{\R^d} |v|^2\varphi_R(v)\;d\mu_0(v)\\
				\qquad\displaystyle=\int_0^t\iiR w \Pi \left(v\varphi_R(v)+|v|^2\frac{\nabla\varphi(v/R)}{R}-v_*\varphi_R(v_*)-|v_*|^2\frac{\nabla\varphi(v_*/R)}{R}\right)\;dM_s(v,v_*)ds.
			\end{array}
		\end{equation}
		We estimate the contribution of $v\phi_R(v) - v_* \phi_R(v_*)$ from the integral in~\eqref{eq:Renergy} using the cancellation from the projection $\Pi[v-v_*]$ to obtain
		$$
		\begin{array}{rcl}
			\displaystyle \left|\int_0^t\iiR w \Pi \left(v\varphi_R(v)-v_*\varphi_R(v_*)\right)\;dM_s\right|&\le& \displaystyle\int_0^t\iint_{(B_R\times B_R)^c}w \left|v\varphi_R(v)-v_*\varphi_R(v_*)\right|\;d|M_s|\\
			&\lesssim&\displaystyle\int_0^t\iint_{(B_R\times B_R)^c}(1+|v|+|v_*|)\;d|M_s|,
		\end{array}
		$$
		where we have used $\gamma\in[-4,-2]$ to bound
		$$
		w \left|v\varphi_R(v)-v_*\varphi_R(v_*)\right|\lesssim\begin{cases}1&|v-v_*|\le 1\\
			|v|+|v_*|& |v-v_*|\ge 1.
		\end{cases}
		$$
		Similarly, using that $\nabla\phi_R$ is supported in $B_{2R}\setminus B_R$ and that $\left|\partial_{v^i}\left\{|v|^2\frac{\partial_{v^j}\varphi(v/R)}{R}\right\}\right|\lesssim 1$ for every index $i,j\in\{1,\dots, d\}$, we obtain that
		$$
		\begin{array}{rcl}
			\displaystyle \left|\int_0^t\iiR w \Pi \left(|v|^2\frac{\nabla\varphi(v/R)}{R}-|v_*|^2\frac{\nabla\varphi(v_*/R)}{R}\right)\;dM_s\right|&\lesssim&\displaystyle \iint_{(B_R\times B_R)^c}1+|v|+|v_*|\;d|M_s|\,
		\end{array},
		$$
		where we have controlled the difference with a mean-value type estimate.
		From the previous bounds, we can use hypothesis \eqref{eq:finitemuM} to take $R\to\infty$ in \eqref{eq:Renergy} and obtain the conservation of energy
		$$
		\lim_{R\to\infty}\int_{\R^d} |v|^2\varphi_R(v)\;d\mu_t(v)=\lim_{R\to\infty}\int_{\R^d} |v|^2\varphi_R(v)\;d\mu_0(v).
		$$
		The proofs for conservation of mass and momentum involve testing the grazing continuity equation against $\phi_R$ and  $v_i\phi_R$ respectively where $v_i$ is the $i$-th component of $v$. For these statements, the case $\gamma\in[-4,-2]$ follows the same as just presented. For $\gamma\in[-2,0]$, the estimates can be more blunt since the weight is no longer singular.
	\end{proof}
	
	\begin{remark}
		Note that as $\gamma$ increases into the range $(-2,0]$, the weight function $w$ starts adding growth so the mean-value type argument in Lemma~\ref{lem:conserve} no longer helps unless more moments of $M$ are assumed than~\eqref{eq:finitemuM}.
		Due to the conservation of mass, the unique weakly* continuous representative $(\tilde{\mu}_t)$ of Lemma~\ref{lem:ctsrep} has the additional property of being weakly continuous in the context of $\mathscr{P}(\mathbb{R}^d)$.
	\end{remark}
	Based on the previous results, we propose the following definition.
	\begin{definition}[Grazing continuity equation]\label{def:GCE}
		For some terminal time $T > 0$, we define $\mathcal{GCE}_T$ to be the set of pairs of measures $(\mu_t,M_t)_{t\in[0,T]}$ satisfying the following:
		\begin{enumerate}
			\item \label{gce:wkcty} $\mu_t \in \mathscr{P}(\R^d)$ is weakly continuous with respect to $t\in[0,T]$. $(M_t)_{t\in[0,T]}$ is a family of Borel measures belonging to $\mathcal{M}^d$.
			\item We have the moment bound
			$$
			\int_0^T \iiR(1+|v|+|v_*|) d|M_t|(v,v_*)dt < \infty.
			$$
			\item \label{gce:eqn} The grazing continuity equation~\eqref{eq:GCE} is satisfied in the distributional sense. That is, for every $\phi \in C_c^\infty((0,T)\times \mathbb{R}^d)$,
			\[
			\int_0^T \int_{\mathbb{R}^d} \partial_t \phi d\mu_t dt + \frac{1}{2}\int_0^T \iint_{\mathbb{R}^{2d}} \tilde{\nabla}\phi dM_t dt = 0,
			\]
			or equivalently for every $\zeta\in C_c^\infty(\Rd)$,
			\[
			\frac{d}{dt}\iR \zeta(v) d\mu_t(v) = \frac{1}{2}\iiR \tn \zeta(v,v_*) dM_t(v,v_*).
			\]
		\end{enumerate}
		For fixed probability measures $\lambda,\nu$, we may also specify the subset $\mathcal{GCE}(\lambda,\nu)$ as those pairs $(\mu,M)\in\mathcal{GCE}_T$ such that $\mu_0 = \lambda,\; \mu_T = \nu$. For $E>0$, we will speak of curves $(\mu,M)\in\mathcal{GCE}_T^{2,E}$ such that
		\[
		\iR |v|^2 d\mu_t(v) \leq E, \quad \forall t\in[0,T].
		\]
	\end{definition}

	\subsection{Action of a curve}
	\label{sec:actfun}
	In this section, we construct the action of a curve under the grazing continuity equation. We introduce the following function $\alpha : \mathbb{R}^d \times \mathbb{R}_{\geq 0} \to [0,\infty]$ by
	\begin{equation*}
		\alpha (u,s) :=\left\{
		\begin{array}{cc}
			\frac{|u|^2}{2s}, 	&s\neq 0 	\\
			0, 	&s=0, u = 0 	\\
			\infty, 	&s=0, u \neq 0
		\end{array}
		\right..
	\end{equation*}
	\begin{remark}\label{lem:propalpha}
		$\alpha$ is lower semi-continuous (lsc), convex, and positively 1-homogeneous.
	\end{remark}
	For fixed $\mu \in \mathscr{P}(\mathbb{R}^d),M\in\mathcal{M}^d$, we consider the tensorized probability measure $\mu\otimes\mu \in\mathscr{P}(\R^d\times\R^d)$ given by $\mu\otimes \mu(dv, \, dv_*) = \mu(dv)\mu(dv_*)$. Define $\tau \in \mathcal{M}$ given by $\tau = \mu\otimes\mu + |M|$ and the decompositions $\mu\otimes\mu = f^1 \tau$ and $M = N\tau$. We define the action functional as
	\begin{equation}\label{eq:defactfun}
		\mathcal{A}(\mu,M) := \iiR \alpha(N,f^1)d\tau.
	\end{equation}
	This is well-defined by the 1-homogeneity of $\alpha$.
	The following lemma establishes a more concrete expression for the action functional.
	\begin{lemma}\label{lem:actfun}
		Let $\mu \in \mathscr{P}(\mathbb{R}^d)$ be absolutely continuous with respect to $\mathcal{L}$ and $\mu = f\mathcal{L}$. Let $M \in \mathcal{M}^d$ be given such that $\mathcal{A}(\mu,M) < \infty$. Then, $M$ is absolutely continuous with respect to $ff_* dv dv_*$ given by some density $U : \mathbb{R}^d\times \mathbb{R}^d \to \mathbb{R}^d$ such that $M = ff_* U dv dv_* = m dvdv_*$ and
		\begin{equation*}
			\mathcal{A}(\mu,M) = \frac{1}{2}\iint_{\mathbb{R}^{2d}} ff_* |U|^2 dv dv_* = \frac{1}{2}\iint_{\mathbb{R}^{2d}} \frac{|m|^2}{ff_*}dvdv_*.
		\end{equation*}
	\end{lemma}
	\begin{proof}
		The proof is identical to~\cite[Lemma 3.6]{erbar_gradient_2016} up to appropriate modifications. Define $\tau \in \mathcal{M}$ by $\tau = \mu\otimes \mu + |M|$ and label the corresponding densities (which may be infinite) $\mu\otimes \mu = g \tau$ and $M = N\tau$. It suffices to show that $M$ is absolutely continuous with respect to $\mu\otimes \mu$ which is the goal of this proof.
			
		Suppose $S\subset \R^{2d}$ is a measurable set such that $\mu\otimes\mu(S) = 0$. This is equivalent to saying $g = 0$ $\tau$-almost everywhere in $S$. Since $\alpha$ is positive, the assumption $\mathcal{A}(\mu,M)<+\infty$ certainly implies $\alpha(N,g)<+\infty$ $\tau$-almost everywhere in $S$. By definition of $\alpha$, we must also have $N = 0$ $\tau$-almost everywhere in $S$ which is equivalent to saying $M(S)=0$.
	\end{proof}
	\begin{lemma}[Lower semi-continuity of action functional]
		\label{lem:lscact}
		The action functional $\mathcal{A}$ as defined in ~\eqref{eq:defactfun} is lower semi-continuous in both arguments. Specifically, if $\mu_n \rightharpoonup \mu$ weakly in $\mathscr{P}(\mathbb{R}^d)$ and $M_n \overset{*}{\rightharpoonup} M$ weakly* in $\mathcal{M}^d$, we have
		\[
		\mathcal{A}(\mu,M) \leq \liminf_{n\to\infty} \mathcal{A}(\mu_n,M_n).
		\]
	\end{lemma}
	\begin{proof}
		This result is an application of the general lsc result in~\cite[Theorem 3.4.3]{buttazzo_semicontinuity_1989} since $\alpha$ satisfies the required convexity, lsc, and homogeneity assumptions by~\Cref{lem:propalpha}.
	\end{proof}
	Another useful property of the action functional is the compactness provided by bounded action. We first state
	\begin{lemma}\label{lem:lemmeasfun}
		Let $F: \mathbb{R}^{2d} \to [0,\infty]$ be measurable and fix any $\mu\in \mathscr{P}(\mathbb{R}^d), \, M \in \mathcal{M}^d$. We have the following bound:
		\begin{equation}
			\label{eq:lemmeasfun}
			\iiR F(v,v_*) d|M|(v,v_*) \leq \sqrt{2}\mathcal{A}(\mu,M)^\frac{1}{2} \left(
			\iiR F(v,v_*)^2 d\mu(v)d\mu(v_*)
			\right)^\frac{1}{2}
		\end{equation}
	\end{lemma}
	\begin{proof}
		This proof follows~\cite[Lemma 3.8]{erbar_gradient_2016}. We assume $\mathcal{A}(\mu,M)<+\infty$ or else~\eqref{eq:lemmeasfun} holds automatically. This implies that whenever $A\subset \R^{2d}$ is a measurable set, $\mu\otimes\mu(A)=0$ if and only if $|M|(A) = 0$. Therefore, in the following computations we are implicitly integrating away from sets of zero $\mu\otimes \mu$-measure. We provide the simple argument by Cauchy-Schwarz for completeness. By considering $\tau = \mu\otimes\mu + |M|$, we estimate
		\begin{align*}
			\iiR Fd|M|(v,v_*) &\leq \iiR F\left|
			\frac{dM}{d\tau}
			\right|d\tau(v,v_*)= \iiR F\left(
			\left|
			\frac{dM}{d\tau}
			\right|\middle/\sqrt{2\frac{d\mu\otimes\mu}{d\tau}}
			\right)\sqrt{2\frac{d\mu\otimes\mu}{d\tau}} d\tau \\
			&\leq \left(
			\iiR \alpha\left(
			\frac{dM}{d\tau},\frac{d\mu\otimes\mu}{d\tau}
			\right)d\tau
			\right)^\frac{1}{2}\left(\iiR
			2F^2d\mu\otimes\mu
			\right)^\frac{1}{2}     \\ 
			&= \sqrt{2}\mathcal{A}(\mu,M)^\frac{1}{2} \left(
			\iiR F(v,v_*)^2 d\mu(v)d\mu(v_*)
			\right)^\frac{1}{2}.
		\end{align*}
	\end{proof}
	
	\begin{remark}
		\label{rk:moretest}
		Suppose we have $\mu_t\in\mathscr{P}(\R^d)$ such that
		$$
		\int_0^Tm_2(\mu_t)dt = \int_0^T\int_{\R^d}\japangle{v}^2\;d\mu_t(v)dt<\infty,
		$$
		then for $M\in\mathcal{M}_T^d$ the previous estimate~\eqref{eq:lemmeasfun} yields
		\begin{equation}
			\label{eq:1momM2mommu}
			\int_0^T\iiR (1+|v| + |v_*|) d|M_t|(v,v_*)dt \lesssim  \int_0^T\mathcal{A}(\mu_t,M_t)^\frac{1}{2} \left(
			1+2\int_{\R^d}|v|^2\;d\mu_t
			\right)^\frac{1}{2}dt.
		\end{equation}
		Therefore, if the integral in time of the second moment of $\mu$ is bounded, then $M$ satisfies the moments conditions \eqref{eq:finitemuM} and the energy is conserved \eqref{lem:conserve}. In the sequel, we will be considering curves that have bounded second moment which guarantee~\eqref{eq:1momM2mommu}.
	\end{remark}
	\begin{proposition}\label{prop:cpctbddact}
		Let $(\mu_t^n,M_t^n)_n$ be a sequence in $\mathcal{GCE}_{T}$ such that $(\mu_0^n)_n$ is tight and we have the following uniform bounds
		\begin{equation}\label{eq:unitintboundact}
			\sup_{n\in\mathbb{N}}\int_0^T \iR |v|^2\;d\mu_t^ndt < \infty\qquad\mbox{and}\qquad\sup_{n\in\mathbb{N}}\int_0^T \mathcal{A}(\mu_t^n,M_t^n)\;dt < \infty.
		\end{equation}
		Then, there exists $(\mu_t,M_t)\in\mathcal{GCE}_{T}$ such that, possibly after extracting a subsequence, we have the following convergences
		\[
		\begin{array}{cl}
			\mu_t^n \rightharpoonup \mu_t 	&\text{weakly in }\mathscr{P}(\mathbb{R}^d), \quad \forall t \in [0,T] 	\\
			M_t^ndt \overset{*}{\rightharpoonup} M_tdt 	&\text{weakly* in }\mathcal{M}_T^d
		\end{array}.
		\]
		Furthermore, along this subsequence we have the following lower semi-continuity
		\[
		\int_0^T \mathcal{A}(\mu_t,M_t)\;dt \leq \liminf_{n\to\infty}\int_0^T \mathcal{A}(\mu_t^n,M_t^n)\;dt.
		\]
	\end{proposition}
	\begin{proof}[Sketch proof]
		This result follows from a similar proof to~\cite[Lemma 4.5]{dolbeault_new_2009} and~\cite[Proposition 3.11]{erbar_gradient_2016} which we sketch. The second moment bound for $\mu^n$ in~\eqref{eq:unitintboundact} produces a limit $\mu$. Recalling the application of~\Cref{lem:lemmeasfun} in~\Cref{rk:moretest}, the bounded action in~\eqref{eq:unitintboundact} and the estimate~\eqref{eq:1momM2mommu} produce a limit $M_tdt$ for a subsequence of $M^n_tdt$. The lower semi-continuity follows from Fatou's lemma and Lemma~\ref{lem:lscact}.
	\end{proof}
	\subsection{Properties of the Landau metric}
	\label{sec:grazdist}
	We define the distance, $d_L$ induced by the action functional on $\mathscr{P}_{2,E}(\Rd)$. Throughout, we will be working in the grazing continuity equation space defined earlier by $\mathcal{GCE}_T^{2,E}$ for $T>0$ some terminal time and $E>0$ any second moment bound.
	
	\begin{definition}
		\label{def:d_L}
		For $\lambda, \nu \in \mathscr{P}_{2,E}(\R^d)$ we define the (square of the) Landau distance by
		\begin{equation}
			\label{eq:d_L}
			d_L^2(\lambda, \nu) := \inf \left\{
			T\int_0^T \mathcal{A}(\mu_t,M_t)dt \, \bigg| \, (\mu,M)\in \mathcal{GCE}_T^{2,E}(\lambda, \nu)
			\right\}.
		\end{equation}
	\end{definition}
Notice this definition is independent of $T>0$ considering the scaling of the grazing collision equation and the 1-homogeneity of $\mathcal{A}$.
	We have an equivalent characterization of $d_L$ which can be seen in other PDE contexts such as~\cite{erbar_gradient_2016,dolbeault_new_2009}.
	\begin{lemma}
		\label{lem:d_Lchar}
		Given $\lambda, \nu \in \mathscr{P}_{2,E}(\R^d)$, we have
		\begin{equation}
			\label{eq:d_Lchar}
			d_L(\lambda, \nu) = \inf\left\{
			\int_0^T \sqrt{\mathcal{A}(\mu_t,M_t)}dt \,\bigg| \, (\mu,M)\in\mathcal{GCE}_T^{2,E}(\lambda, \nu)
			\right\}.
		\end{equation}
	\end{lemma}
	\begin{proof}
		This proof uses the same reparameterisation technique in~\cite[Theorem 5.4]{dolbeault_new_2009}.
	\end{proof}
	\begin{proposition}[Minimizing curve]
		\label{prop:mincurve}
		Suppose that $\mu_0,\mu_1\in\mathscr{P}_{2,E}(\R^d)$ are probability measures such that $d_L(\mu_0,\mu_1)<\infty$. Then there exists a curve $(\mu,M)\in\mathcal{GCE}_1^{2,E}(\mu_0,\mu_1)$ attaining the infimum of~\eqref{eq:d_L} (equivalently, also~\eqref{eq:d_Lchar}) and $\mathcal{A}(\mu_t,M_t) = d_L^2(\mu_0,\mu_1)$ for almost every $t\in[0,1]$.
	\end{proposition}
	\begin{proof}
		This result follows from the direct method of calculus of variations where the lower semicontinuity comes from Proposition~\ref{prop:cpctbddact}.
	\end{proof}
	\begin{proof}[Proof of Theorem~\ref{them:existdL}]
		We prove the statements in exactly the order they are presented in the theorem, starting with the properties of the proposed Landau distance as a metric. The positivity of $d_L$ follows from the positivity of $\alpha$. We now check that $d_L$ satisfies the properties of a metric.
		
		\textitunder{$d_L$ distinguishes points}
		\\
		Fix $\mu_0,\mu_1\in \mathscr{P}_{2,E}(\Rd)$, we check that $d_L(\mu_0,\mu_1) = 0 \iff \mu_0 = \mu_1$. Suppose that $d_L(\mu_0,\mu_1) = 0$. By Proposition~\ref{prop:mincurve} we can find $(\mu,M)\in \mathcal{GCE}_1^{2,E}(\mu_0,\mu_1)$ which is a minimizing curve and moreover $0 = d_L(\mu_0,\mu_1) = \mathcal{A}(\mu_t,M_t)$ implies $M=0$. The grazing continuity equation reduces to $\partial_t\mu_t=0$ which implies $\mu_t$ is constant in time.
		
		The converse statement follows similarly by pairing the constant curve $\mu : t\mapsto \mu_0=\mu_1$ with the zero measure so that $(\mu,0) \in \mathcal{GCE}_1^{2,E}(\mu_0,\mu_1)$.
		
		\textitunder{Symmetry}
		\\
		Symmetry follows because time can be reversed for every curve. For instance, if $(\mu,M) \in \mathcal{GCE}_T^{2,E}({\mu_0},{\mu_1})$, then one can check that the pair
		\[
		\mu^r :t\mapsto \mu(T-t), \quad M^r :t \mapsto -M(T-t)
		\]
		belong to $\mathcal{GCE}_T^{2,E}(\mu_1,\mu_0)$ with the same action.
		
		\textitunder{Triangle inequality}
		\\
		We sketch the argument using a glueing lemma as in~\cite[Lemma 4.4]{dolbeault_new_2009}. Let $\mu^0,\mu^1,\mu^2\in\mathscr{P}_{2,E}(\Rd)$ be such that $d_L(\mu^0,\mu^1)<\infty$ and  $d_L(\mu^1,\mu^2)<\infty$.  If not, $d_L(\mu^0,\mu^2)\leq d_L(\mu^0,\mu^1) + d_L(\mu^1,\mu^2)$ holds trivially. By Proposition~\ref{prop:mincurve}, we can find minimizing curves connecting these probability measures
		\[
		\left\{
		\begin{array}{cl}
			(\mu^{0\to1},M^{0\to1}) &\in\mathcal{GCE}_1^{2,E}({\mu^0},{\mu^1}) 	\\
			(\mu^{1\to2},M^{1\to2}) &\in\mathcal{GCE}_1^{2,E}({\mu^1},{\mu^2}) 
		\end{array}
		\right\}.
		\]
		Their concatenation from time 0 to 1 is given by
		\[
		\mu_t:= \left\{
		\begin{array}{cc}
			\mu_{2t}^{0\to1}, &0 \leq t \leq 1/2 	\\
			\mu_{2(t-1/2)}^{1\to2}, 	&1/2\leq t \leq 1
		\end{array}
		\right., \quad M_t:= \left\{
		\begin{array}{cc}
			2M_{2t}^{0\to1}, &0 \leq t \leq 1/2 	\\
			2M_{2{t-1/2}}^{1\to2}, 	&1/2< t \leq 1
		\end{array}
		\right..
		\]
		One can check that $(\mu,M)\in\mathcal{GCE}_1^{2,E}(\mu^0,\mu^2)$, so it is an admissible competitor in the computation of $d_L(\mu^0,\mu^2)$. By looking at the action on the different time pieces, we obtain
		\begin{align*}
			d_L(\mu^0,\mu^2) \leq \int_0^1 \sqrt{\mathcal{A}(\mu_t,M_t)}dt = d_L(\mu^0,\mu^1)+d_L(\mu^1,\mu^2).
		\end{align*}
		
		\textitunder{$d_L$-convergence/boundedness implies weak convergence/compactness}
		\\
		Fix $\mu^n,\mu^\infty\in\mathscr{P}_{2,E}$ for $n\in\mathbb{N}$ be such that $d_L(\mu^\infty,\mu^n) \to 0$ as $n\to \infty$. By Proposition~\ref{prop:mincurve}, take minimizing curves $(\nu^n,M^n)\in\mathcal{GCE}_1^{2,E}(\mu^\infty,\mu^n)$ such that
		\[
		d_L(\mu^\infty,\mu^n) = \sqrt{\mathcal{A}(\nu_t^n,M_t^n)}, \quad \text{a.e. }t\in[0,1].
		\]
		By compactness in Proposition~\ref{prop:cpctbddact}, there are limits $(\nu,M)\in\mathcal{GCE}_1^{2,E}$ such that $\nu^n \rightharpoonup \nu$ and $M^n \overset{*}{\rightharpoonup}M$ up to a subsequence. Moreover, the lower semicontinuity in Proposition~\ref{prop:cpctbddact} gives
		\[
		\mathcal{A}(\nu_t,M_t) \leq \liminf_{n\to\infty}\mathcal{A}(\nu_t^n,M_t^n) = 0,
		\]
		hence $M=0$ so that $\nu$ is a constant in time. Since $\nu(0) = \mu^\infty$, this implies $\mu^\infty = \nu(1) = \lim_{n\to \infty} \mu^n$ which establishes the weak convergence.
		\\
		\textitunder{$(\mathscr{P}_\tau,d_L)$ is a complete geodesic space}
		\\
		We start with the geodesic property from completely analogous arguments to Erbar~\cite{erbar_gradient_2016}, the remaining statement that $\mathscr{P}_\tau$ equipped with $d_L$ is a complete geodesic space follows. Fix $\tau\in\mathscr{P}_{2,E}(\R^d)$ with $\mu_0,\mu_1\in\mathscr{P}_\tau$, the triangle inequality ensures $d_L(\mu_0,\mu_1)<\infty$ so Proposition~\ref{prop:mincurve} guarantees the existence of a minimizing curve $(\mu,M)\in\mathcal{GCE}_1^{2,E}({\mu_0},{\mu_1})$. One easily sees that this also induces a minimizing curve for intermediate times. More precisely, for every $0 \leq r \leq s \leq 1$, we have that $(t\mapsto\mu_{t+r},t\mapsto M_{t+r})\in\mathcal{GCE}_{s-r}^{2,E}(\mu_r,\mu_s)$ also minimizes $d_L(\mu_r,\mu_s)$.
		
		To show completeness, let $(\mu^n)_{n\in\mathbb{N}}$ be a Cauchy sequence in $\mathscr{P}_\tau$. The sequence is certainly $d_L$-bounded so by Proposition~\ref{prop:cpctbddact}, we can find, up to extraction of a weakly convergent subsequence, $\mu^\infty\in\mathscr{P}_{2,E}(\R^d)$ such that $\mu^n\rightharpoonup\mu^\infty$ in $\mathscr{P}_{2,E}(\R^d)$. Lower semi-continuity of $d_L$ and the Cauchy property of the subsequence give
		\[
		d_L(\mu^n,\mu^\infty) \leq \liminf_{m\to\infty}d_L(\mu^n,\mu^m) \to 0, \quad \text{as }n\to\infty.
		\]
		For any $n\in\mathbb{N}$ the triangle inequality gives
		\[
		d_L(\mu^\infty,\tau) \leq d_L(\mu^\infty,\mu^n) + d_L(\mu^n,\tau) < \infty,
		\]
		So $\mu^\infty\in\mathscr{P}_\tau$. 
	\end{proof}

	\begin{proposition}[Metric derivative]
		\label{prop:metvel}
		A curve $(\mu_t)_{t\in[0,T]}\subset \mathscr{P}_{2,E}(\R^d)$ is absolutely continuous with respect to $d_L$ if and only if there exists a Borel family $(M_t)_{t\in[0.T]}$ belonging to $\mathcal{M}_T^d$ such that $(\mu,M)\in\mathcal{GCE}_T^{2,E}$ with the property that
		\begin{equation*}
			\int_0^T\sqrt{\mathcal{A}(\mu_t,M_t)}dt<\infty.
		\end{equation*}
		In this equivalence, we have a bound on the metric derivative
		\begin{equation*}
			\lim_{h\downarrow0}\frac{d_L^2(\mu_{t+h},\mu_t)}{h^2} =: |\dot{\mu}|^2(t) \leq \mathcal{A}(\mu_t,M_t), \quad \text{a.e. }t\in(0,T).
		\end{equation*}
		Furthermore, there exists a unique Borel family $(\tilde{M}_t)_{t\in[0,T]}$ belonging to $\mathcal{M}^d$ which is characterized by
		\[
		M_t=U \mu_t\otimes\mu_t\qquad\mbox{and}\qquad  U\in T_\mu:= \overline{\{\tilde{\nabla}\phi \,| \, \phi \in C_c^\infty(\mathbb{R}^d)\}}^{L^2(\mu_t\otimes\mu_t)}
		\]
		such that $(\mu,\tilde{M})\in\mathcal{GCE}_T^E(\mu_0,\mu_T)$ where we have equality:
		\begin{equation*}
			|\dot{\mu}|^2(t) = \mathcal{A}(\mu_t,\tilde{M}_t), \quad \text{a.e. }t\in(0,T).
		\end{equation*}
	\end{proposition}
	\begin{proof}
		The argument follows exactly as in~\cite[Theorem 5.17]{dolbeault_new_2009}.
	\end{proof}

	\section{Energy dissipation equality}
	\label{sec:EDE}
	The goal in this section is to prove Theorem~\ref{thm:epsequiv} which states that the notions of gradient flow solutions coincide with $\epsilon$-solutions to the Landau equation. To fix ideas, we recall the regularized entropy functionals acting on probability measures
	\[
	\mathcal{H}_\epsilon [\mu] = \int_{\mathbb{R}^d} (\mu*G^\epsilon)(v)\log(\mu*G^\epsilon)(v)dv,
	\]
	with $G^\epsilon(v)$ given by
	\[
	G^\epsilon(v) = \epsilon^{-d}C_d \exp \left\{
	-\japangle{\frac{v}{\epsilon}}
	\right\}.
	\]
	The crucial ingredient to prove Theorem~\ref{thm:epsequiv} is the following
	\begin{proposition}[Chain Rule $\epsilon$]\label{prop:strongupgradeps}
		Fix $\gamma\in[-4,0]$ and suppose $(\mu,M)\in\mathcal{GCE}_{T}^{2,E}$ and
		\begin{equation*}
			\int_0^T \mathcal{A}(\mu_t,M_t)dt < \infty.
		\end{equation*}
		Then, $\sup_{t\in[0,T]}\mathcal{H}_\epsilon[\mu_t]< \infty$ and the `chain rule' holds
		\begin{equation}\label{eq:FTCentropy}
			\mathcal{H}_\epsilon[\mu_r] - \mathcal{H}_\epsilon[\mu_s] = \frac{1}{2}\int_s^r \iiR \tn \left[
			\firstvar{\mathcal{H}_\epsilon}{\mu}
			\right]\cdot dM_t dt, \quad \forall 0\leq s \leq r \leq T.
		\end{equation}
	\end{proposition}
	\begin{remark}\label{rk:stronguppergradient}
		Recall the expression for the dissipation
		$$
		D_\epsilon[\mu]=\frac{1}{2}\iiR \left|\tn \left[
		\firstvar{\mathcal{H}_\epsilon}{\mu}
		\right]\right|^2 d\mu(v)d\mu(v_*).
		$$
		Using a time integrated version of Lemma~\ref{lem:lemmeasfun}, we have the estimate
		\begin{equation*}
			\frac{1}{2}\int_s^r \iiR \left|\tn \left[
			\firstvar{\mathcal{H}_\epsilon}{\mu}
			\right]\right|\cdot d|M_t|(v,v_*) dt \leq
			\int_s^r \mathcal{A}(\mu_t,M_t)^\frac{1}{2} D_\epsilon[\mu_t]^\frac{1}{2}\;dt.
		\end{equation*}
		Therefore, under the hypothesis of Proposition~\ref{prop:strongupgradeps}, we have that
		$$
		|\mathcal{H}_\epsilon(\mu_r)-\mathcal{H}_\epsilon(\mu_r)|\le \int_s^r |\dot{\mu}|(t) D_\epsilon[\mu_t]^\frac{1}{2}\;dt,
		$$
		which implies that $D_\epsilon[\mu_t]^\frac{1}{2}$ is a strong upper gradient of $\mathcal{H}_\epsilon$, see Definition~\ref{defn:sug}.
	\end{remark}
	Taking Proposition~\ref{prop:strongupgradeps} for granted, we can prove Theorem~\ref{thm:epsequiv}.
	\begin{proof}
		[Proof of Theorem~\ref{thm:epsequiv}]
		Throughout, $\mu=f \mathcal{L}$ is a curve of probability measures with uniformly bounded second moment.
		
		\textitunder{Weak $\epsilon$-solution $\implies$ Curve of maximal slope}
		\\
		Consider $f$ an $\epsilon$-solution to the Landau equation. Define $m = - ff_* \tn \firstvar{\mathcal{H}_\epsilon}{f}$ so that the pair of measures $(\mu = f \mathcal{L}, M  = m \mathcal{L}\otimes \mathcal{L})$ therefore belong to $\mathcal{GCE}_{T}^E$. Indeed, the distributional grazing continuity equation from Definition~\ref{def:GCE} is precisely the weak $\epsilon$ Landau equation. Based on the definition of $M$ and the finite $\mathcal{H}_{\epsilon}$ dissipation, we have the bound
		\[
		\int_0^T\mathcal{A}(\mu_t,M_t)dt = \int_0^TD_\epsilon (f_t) dt< \infty,
		\]
		which implies the weak continuity of $\mu$. By Proposition~\ref{prop:metvel}, we have
		\begin{equation*}
			|\dot{\mu}|^2(t) = \mathcal{A}(\mu_t,M_t) = D_\epsilon(f_t)<\infty, \quad \text{a.e. }t\in[0,T].
		\end{equation*}
		Using Proposition~\ref{prop:strongupgradeps}, we have for any $0 \leq s \leq r \leq T$
		\begin{align*}
			\mathcal{H}_{\epsilon}[\mu_r] - \mathcal{H}_{\epsilon}[\mu_s] + \frac{1}{2}\int_s^rD_\epsilon(\mu_t)dt + \frac{1}{2}\int_s^r |\dot{\mu}|^2(t)dt 	\leq 0.
		\end{align*}
		According to Definition~\ref{defn:cms}, this is the curve of maximal slope property.
		
		\textitunder{Curve of maximal slope $\implies$ weak $\epsilon$-solution}
		\\
		Assume that $\mu = f\mathcal{L}$ is a curve of maximal slope for $\mathcal{H}_{\epsilon}$ with respect to the upper gradient $\sqrt{D_\epsilon}$. Since $\mu$ is absolutely continuous with respect to $d_L$, Proposition~\ref{prop:metvel} guarantees existence of a unique curve $M : t\in[0,T]\mapsto M_t \in \mathcal{M}^d$ such that $\int_0^T \sqrt{\mathcal{A}(\mu_t,M_t)}dt < \infty$ and $|\dot{\mu}|^2(t) = \mathcal{A}(\mu_t,M_t)$ a.e. $t\in[0,T]$. Furthermore, the pair $(\mu,M)\in\mathcal{GCE}_{T}^E$. According to Lemma~\ref{lem:actfun}, let $M = m\mathcal{L}\otimes \mathcal{L}$ for some measurable function $m$. We apply the chain rule~\eqref{eq:FTCentropy} with Cauchy-Schwarz and Young's inequalities with minus signs in the follow computations.
		\begin{align*}
			\mathcal{H}_{\epsilon}[f_T] - \mathcal{H}_{\epsilon}[f_0] &= \frac{1}{2}\int_0^T \iiR \tn \firstvar{\mathcal{H}_\epsilon}{f} \cdot m dvdv_* dt 	\\
			&\geq -\frac{1}{2}\int_0^T \left(
			\iiR ff_* \left|
			\tn \firstvar{\mathcal{H}_\epsilon}{f}
			\right|^2dvdv_*
			\right)^\frac{1}{2}\left(
			\iiR \frac{|m|^2}{ff_*}dvdv_*
			\right)^\frac{1}{2}dt 	\\
			&\geq -\frac{1}{2} \int_0^T 
			\left(\frac{1}{2}\iiR ff_* \left|
			\tn \firstvar{\mathcal{H}_\epsilon}{f}
			\right|^2dvdv_*\right)
			dt
			-\frac{1}{2}\int_0^T\left(\frac{1}{2}
			\iiR \frac{|m|^2}{ff_*}dvdv_*\right)
			dt 	\\
			&= -\frac{1}{2}\int_0^T D_\epsilon(f_t)dt - \frac{1}{2}\int_0^T |\dot{f}|^2(t)dt.
		\end{align*}
		All the inequalities in the calculations above are actually equalities owing to the fact that $\mu$ is a curve of maximal slope. In particular, since we have the equality in the Young's inequality, this implies that $\frac{m}{\sqrt{ff_*}} = -\sqrt{ff_*}\tn \firstvar{\mathcal{H}_\epsilon}{f}$. As in the previous direction, the weak $\epsilon$ Landau equation coincides with the grazing continuity equation when $m$ is equal to $-ff_* \tn \firstvar{\mathcal{H}_\epsilon}{f}$.
	\end{proof}
	The rest of this section is devoted to proving Proposition~\ref{prop:strongupgradeps}. We need some lemmata to establish crucial estimates. The following result is a variation of~\cite[Lemma 2.6]{CC92}.
	\begin{lemma}[Carlen-Carvalho~\cite{CC92}]
		\label{lem:carlen}
		Let $\mu$ be a probability measure on $\mathbb{R}^d$ with finite second moment/energy, $m_2(\mu) \leq E$ for $E>0$. Then, for every $\epsilon>0$, there exists a constant $C=C(\epsilon,E)>0$ such that
		\[
		|\log (\mu*G^\epsilon)(v)| \leq C\japangle{\frac{v}{\epsilon}}.
		\]
	\end{lemma}
	\begin{proof} 
		Starting with an upper bound, we easily see
		\[
		\mu*G^\epsilon (v) = \int_{\mathbb{R}^d} G^\epsilon(v-v')d\mu(v') \lesssim_\epsilon 1.
		\]
		Turning to the lower bound, we cut off the integration domain to $|v'|\leq R$, for some $R>0$ to be chosen later. We estimate, for $\epsilon>0$ small enough
		\[
		\japangle{\frac{v-v'}{\epsilon}} = \sqrt{1+\left|
			\frac{v-v'}{\epsilon}
			\right|^2} \leq \sqrt{1+2\left|
			\frac{v}{\epsilon}
			\right|^2 + 2\left(
			\frac{R}{\epsilon}
			\right)^2} \leq \sqrt{2}\left(
		\japangle{\frac{v}{\epsilon}} + \japangle{\frac{R}{\epsilon}}
		\right).
		\]
		This is substituted into $G^\epsilon(v-v')$ to obtain
		\[
		\mu*G^\epsilon(v) \geq \int_{|v'|\leq R} G^\epsilon(v-v')d\mu(v') \gtrsim_\epsilon \exp\left\{
		-\sqrt{2}\left(\japangle{\frac{v}{\epsilon}} + \japangle{\frac{R}{\epsilon}}\right)
		\right\}\int_{|v'|\leq R} d\mu(v').
		\]
		At this point, we appeal to Chebyshev's inequality to see
		\[
		\int_{|v'|\leq R} d\mu(v') = 1-\int_{|v'|\geq R} d\mu(v') \geq 1- \frac{1}{R^2}\int_{|v'|\geq R} |v'|^2 d\mu(v').
		\]
		We can now choose, for example, large $R$ such that $1-\frac{E}{R^2} \geq \frac{1}{2}$ to uniformly lower bound the integral $\int_{|v'|\leq R}d\mu(v')$ away from $0$ and then conclude the result after applying logarithms.
	\end{proof}
	\begin{lemma}[log-derivative estimates]
		\label{lem:estlogdiff}
		For fixed $\epsilon>0$ we have the formula
		\begin{equation}
			\label{eq:diffGseps}
			\nabla G^\epsilon(v) = -\frac{1}{\epsilon}\japangle{\frac{v}{\epsilon}}^{-1}G^\epsilon(v) \frac{v}{\epsilon}.
		\end{equation}
		For $\mu\in\mathscr{P}(\Rd)$, denoting $\partial^i = \frac{\partial}{\partial v^i}$ and $\partial^{ij} = \frac{\partial^2}{\partial{v^i}\partial{v^j}}$, we obtain
		\begin{equation}
			\label{eq:extlogdiffsgeq1}
			\left|
			\nabla \log (\mu*G^\epsilon)(v)
			\right|\leq \frac{1}{\epsilon}, \quad \left|
			\partial^{ij} \log (\mu*G^\epsilon)(v)
			\right|\leq \frac{4}{\epsilon^2}.
		\end{equation}
	\end{lemma}
	\begin{proof}
		Equation~\eqref{eq:diffGseps} is a direct computation after noticing
		\[
		\frac{\nabla G^\epsilon}{G^\epsilon} = \nabla \log G^\epsilon = \nabla \left(
		-\japangle{\frac{v}{\epsilon}} + const.
		\right) = -\frac{1}{\epsilon}\japangle{\frac{v}{\epsilon}}^{-1}\frac{v}{\epsilon}.
		\]
		The first order log-derivative estimate of~\eqref{eq:extlogdiffsgeq1} is calculated using formula~\eqref{eq:diffGseps} to obtain
		\begin{align*}
			&\quad |\nabla (\mu*G^\epsilon)(v)| = |\mu * \nabla G^\epsilon(v)| \leq \frac{1}{\epsilon}\iR \japangle{\frac{v-v'}{\epsilon}}^{-1}\left|\frac{v-v'}{\epsilon}\right| G^\epsilon(v-v')d\mu(v') 	\\
			&\leq \frac{1}{\epsilon}\iR G^\epsilon(v-v')d\mu(v') = \frac{1}{\epsilon} (\mu*G^\epsilon)(v).
		\end{align*}
		For the second order, we first look at $\partial^{ij}\mu*G^\epsilon$ which can be computed with the help of~\eqref{eq:diffGseps}
		\begin{align*}
			&|\partial^{ij}\mu*G^\epsilon(v)| = \left|\partial^i \left(
			-\frac{1}{\epsilon}\iR \japangle{\frac{v-v'}{\epsilon}}^{-1}\frac{v^j-v'^j}{\epsilon}G^\epsilon(v-v')d\mu(v')
			\right)\right| =	\\
			& \left|\frac{1}{\epsilon^2}\iR \left(
			\japangle{\frac{v-v'}{\epsilon}}^{-3}\frac{v^i-v'^i}{\epsilon}\frac{v^j-v'^j}{\epsilon} + \delta^{ij}\japangle{\frac{v-v'}{\epsilon}}^{-1} \right.\right.\\
			&\qquad\qquad\qquad\qquad\qquad\qquad\qquad\qquad\qquad \left.\left.- \japangle{\frac{v-v'}{\epsilon}}^{-2}\frac{v^i-v'^i}{\epsilon}\frac{v^j-v'^j}{\epsilon}
			\right) G^\epsilon(v-v')d\mu(v')\right| 	\\
			&\leq \frac{3}{\epsilon^2}\mu*G^\epsilon(v).
		\end{align*}
		Combining this estimate with the previous first order one, we have
		\begin{align*}
			&\quad \left|\partial^{ij} \log (\mu*G^\epsilon)(v)\right| = \left|\frac{\partial^{ij}\mu*G^\epsilon}{\mu*G^\epsilon} - \frac{(\partial^i\mu*G^\epsilon)(\partial^j\mu*G^\epsilon)}{(\mu*G^\epsilon)^2}\right|	 
			\leq \frac{4}{\epsilon^2}.
		\end{align*}
	\end{proof}
\begin{lemma}
	\label{lem:esttnfirstvar}
	Fix $\epsilon>0$ and $\gamma\in[-4,0]$ with $\mu\in\mathscr{P}_{2,E}(\Rd)$ for some $E>0$. We have
	\begin{enumerate}
		\item \textitunder{Moderately soft case $\gamma\in[-2,0]$:}
		\[
		\left|\tn\frac{\delta\mathcal{H}_\epsilon}{\delta \mu}\right| = \left| \tn [G^\epsilon*\log (\mu*G^\epsilon)](v,v_*)\right| \lesssim_\epsilon |v|^{1+\frac{\gamma}{2}}+ |v_*|^{1+\frac{\gamma}{2}}.
		\]
		\item \textitunder{Very soft case $\gamma\in[-4,-2]$:}
		\[
		\left|\tn\frac{\delta\mathcal{H}_\epsilon}{\delta \mu}\right|  \lesssim_\epsilon 1.
		\]
	\end{enumerate}
	In particular, it holds
	\[
	\iiR \left|\tn\frac{\delta\mathcal{H}_\epsilon}{\delta \mu}\right|^2 d\mu(v)d\mu(v_*) \leq E.
	\]
\end{lemma}
\begin{proof}
	We develop the expression for $\tn \frac{\delta \mathcal{H}_\epsilon}{\delta \mu}$ in integral form to be used throughout this proof.
	\begin{align}
		\label{eq:tnfattail}
		\begin{split}
			&\quad \tn \frac{\delta \mathcal{H}_\epsilon}{\delta \mu} = \tn G^\epsilon * \log(\mu*G^\epsilon)(v,v_*) 	\\
			&= |v-v_*|^{1+\frac{\gamma}{2}} \Pi[v-v_*](\nabla_v G^\epsilon*\log(\mu*G^\epsilon)(v) - \nabla_{v_*}G^\epsilon*\log (\mu*G^\epsilon)(v_*)) 	\\
			&= |v-v_*|^{1+\frac{\gamma}{2}}\Pi[v-v_*]\iR G^\epsilon(v')\left(
			\frac{\nabla \mu*G^\epsilon}{\mu*G^\epsilon}(v-v') - \frac{\nabla\mu*G^\epsilon}{\mu*G^\epsilon}(v_*-v')
			\right)dv'.
		\end{split}
	\end{align}
	\begin{enumerate}
		\item \textitunder{Moderately soft case $\gamma\in[-2,0]$:} We use (a concave version of) the triangle inequality (valid since $1+\frac{\gamma}{2}\geq 0$) and the first estimate of~\eqref{eq:extlogdiffsgeq1} to bound the last line of~\eqref{eq:tnfattail}
		\begin{align*}
			&\quad \left|
			\tn \frac{\delta \mathcal{H}_\epsilon}{\delta \mu}
			\right| \leq 2^{1+\frac{\gamma}{2}}(|v|^{1+\frac{\gamma}{2}}+ |v_*|^{1+\frac{\gamma}{2}})\frac{2}{\epsilon}\iR G^\epsilon(v')dv' \lesssim_\epsilon|v|^{1+\frac{\gamma}{2}} + |v_*|^{1+\frac{\gamma}{2}}.
		\end{align*}
		\item \textitunder{Very soft case $\gamma\in[-4,-2]$:} We perform estimates in two cases, the far field $|v-v_*| \geq 1$ and near field $|v-v_*|\leq 1$.
		\\
		\textitunder{$|v-v_*|\geq 1$:}
		\\
		In the far field, we have $|v-v_*|^{1+\frac{\gamma}{2}} \leq 1$ hence we can brutally estimate~\eqref{eq:tnfattail} using again the first estimate of~\eqref{eq:extlogdiffsgeq1} to obtain, similar to the moderately soft case, the estimate
		\[
		\left|
		\tn \frac{\delta \mathcal{H}_\epsilon}{\delta \mu}
		\right| \leq \frac{2}{\epsilon}.
		\]
		\textitunder{$|v-v_*|\leq 1$:}
		\\
		We can remove the singularity from the weight with a mean-value estimate and the second estimate of~\eqref{eq:extlogdiffsgeq1}
		\[
		\left|
		\frac{\nabla \mu*G^\epsilon}{\mu*G^\epsilon}(v-v') - \frac{\nabla\mu*G^\epsilon}{\mu*G^\epsilon}(v_*-v')
		\right| \leq \sup_{i,j=1,\dots,d}\left|\left|\partial^i\left(
		\frac{\partial^j\mu*G^\epsilon}{\mu*G^\epsilon}
		\right)\right|\right|_{L^\infty} |v-v_*| \leq\frac{4}{\epsilon^2}|v-v_*|.
		\]
		Inserting this into~\eqref{eq:tnfattail}, we have
		\[
		\left|
		\tn \frac{\delta \mathcal{H}_\epsilon}{\delta \mu}
		\right| \leq \frac{4}{\epsilon^2}|v-v_*|^{2+\frac{\gamma}{2}} \iR G^\epsilon(v')dv' \leq \frac{4}{\epsilon^2}.
		\]
	\end{enumerate}
\end{proof}
\begin{remark}
	\label{rk:Glogdiff}
	Originally, we considered the general family of convolution kernels $G^{s,\epsilon}$ described in Section~\ref{sec:notation}. Besides the context of the Landau equation, Lemma~\ref{lem:estlogdiff} (excluding the second order log-derivative estimate) can be generalized to this family of $s$-order tailed exponential distributions with additional moment assumptions on $\mu$. In particular, equations~\eqref{eq:diffGseps} and~\eqref{eq:extlogdiffsgeq1} (for $s\geq 1$) become
	\[
	\frac{\nabla G^{s,\epsilon}}{G^{s,\epsilon}}(v) = - \frac{s}{\epsilon}\japangle{\frac{v}{\epsilon}}^{s-2}\frac{v}{\epsilon}, \quad \frac{|\nabla (\mu * G^{s,\epsilon})|}{\mu*G^{s,\epsilon}}(v) \lesssim \frac{1}{\epsilon^s}\japangle{v}^{s-1}.
	\]
	Since Maxwellians are known to be stationary solutions for the Landau equation, we wanted to perform the regularization with $s=2$. However, the analogous estimates of Lemma~\ref{lem:estlogdiff} for $s=2$ are not sufficient for Lemma~\ref{lem:esttnfirstvar} in the $\mathscr{P}_2$ framework. For example, in the moderately soft potential case, the estimate reads
	\[
	\left|\tn \firstvar{\mathcal{H}_{2,\epsilon}}{\mu}\right| \lesssim_\epsilon \japangle{v}^{2+\frac{\gamma}{2}} + \japangle{v_*}^{2+\frac{\gamma}{2}}\notin L^2(\mu\otimes\mu).
	\]
	However, there is one value of $\gamma=-2$ for which the estimates hold when using a Maxwellian regularization kernel $G^{2,\epsilon}$. A restriction to $\mathscr{P}_4$ resolves the issue mentioned above for the moderately soft potential case, but then a fourth moment propagation is needed which we did not pursue. A similar issue is present in the very soft potential case.
\end{remark}
\begin{proof}[Proof of Proposition~\ref{prop:strongupgradeps}]
	To prove equation~\eqref{eq:FTCentropy}, our strategy is to regularize the pair $(\mu,M)$ in time with parameter $\delta>0$ and differentiate the regularization. Then we obtain uniform bounds in $\delta$ needed to take the limit $\delta\to 0$. 
	
	\textitunder{Finite regularized entropy}
	\\
	We have the following chain of inequalities
	\[
	\mathcal{H}_\epsilon[\mu_t] = \iR (\mu_t*G^\epsilon)(v) \log (\mu_t*G^\epsilon)(v) dv \lesssim_{\epsilon,E} \iR (\mu_t*G^\epsilon)(v) \japangle{v}dv\lesssim_\epsilon 1+E.
	\]
	The first inequality comes from Lemma~\ref{lem:carlen} because $\log(\mu_t*G^\epsilon)$ has linear growth (uniform in time) while in the second inequality, one realises that $\mu_t*G^\epsilon$ has as many moments as $\mu_t$ with computable constants.
	
	\textitunder{Time regularization with $\delta>0$}
	\\
	Without loss of generality, let $\mu$ be the weakly time continuous representative (Lemma~\ref{lem:ctsrep}) and $M$ be the optimal grazing rate (Proposition~\ref{prop:metvel}) achieving the finite distance $d_L$.
	We first regularize the pair $(\mu,M)$ in time for a fixed parameter $\delta>0$ as follows. Take $\eta \in C_0^\infty(\mathbb{R})$ with the following properties
	\[
	\text{supp} \,\eta \subset (-1,1), \quad \eta \geq 0, \quad \eta(t) = \eta(-t), \quad \int_{-1}^1 \eta (t) dt = 1.
	\]
	We define the following measures for $t\in[0,T]$, by taking convex combinations
	\[
	\mu_t^\delta := 
	\int_{-1}^1 \eta(t')\mu_{t-\delta t'}dt'
	, \quad M_t^\delta := 
	\int_{-1}^1\eta (t')M_{t-\delta t'}dt'
	.
	\]
	
	Here, we constantly extend the measures in time. That is, if $t-\delta t' \in [-\delta,0]$, we treat $\mu_{t-\delta t' } = \mu_0, M_{t-\delta t' } = 0$. For the other end point, if $t-\delta t' \in [T,T+\delta]$, we set $\mu_{t-\delta t' } = \mu_T, M_{t-\delta t' } = 0$. This transformation is stable so that $(\mu^\delta, M^\delta) \in \mathcal{GCE}_T$ and in particular, the distributional grazing continuity equation holds
	\[
	\partial_t \mu_t^\delta + \frac{1}{2}\tilde{\nabla}\cdot M_t^\delta = 0.
	\]
	We derive equation~\eqref{eq:FTCentropy} using this regularized grazing continuity equation.
	Consider
	\[
	\mathcal{H}_\epsilon[\mu_t^\delta] = \int_{\mathbb{R}^d} (\mu_t^\delta * G^\epsilon)(v) \log (\mu_t^\delta*G^\epsilon)(v)dv,
	\]
	which we differentiate with respect to $t$ by appealing to the Dominated Convergence Theorem. Firstly, due to the time regularization, we have
	\[
	\partial_t\left\{
	(\mu_t^\delta*G^\epsilon)\log (\mu_t^\delta*G^\epsilon)
	\right\} = \left[
	(\partial_t\mu_t^\delta) * G^\epsilon
	\right](\log (\mu_t^\delta*G^\epsilon) + 1).
	\]
	The $L_v^1$ bound is obtained on the following difference quotient for a fixed time step $h>0$
	\begin{align*}
		&\quad \left|
		\frac{1}{h} [(\mu_{t+h}^\delta*G^\epsilon)\log(\mu_{t+h}^\delta*G^\epsilon) - (\mu_t^\delta*G^\epsilon)\log(\mu_t^\delta*G^\epsilon)]
		\right|	 	\\
		&\leq \frac{1}{h}\left|
		(\mu_{t+h}^\delta*G^\epsilon) - (\mu_t^\delta*G^\epsilon)
		\right|\sup_{s\in[t,t+h]}\left|
		\log(\mu_s^\delta*G^\epsilon) + 1
		\right|.
	\end{align*}
	where we have used the Mean Value theorem with the chain rule. Applying Lemma~\ref{lem:carlen}, we obtain
	\[\left|
	\frac{1}{h} [(\mu_{t+h}^\delta*G^\epsilon)\log(\mu_{t+h}^\delta*G^\epsilon) - (\mu_t^\delta*G^\epsilon)\log(\mu_t^\delta*G^\epsilon)]
	\right|
	\lesssim_{\epsilon,E}\frac{1}{h}\left|
	(\mu_{t+h}^\delta*G^\epsilon) - (\mu_t^\delta*G^\epsilon)
	\right|\japangle{v}.
	\]
	We apply the Mean Value Theorem on the difference quotient again to get
	\[\left|
	\frac{1}{h} [(\mu_{t+h}^\delta*G^\epsilon)\log(\mu_{t+h}^\delta*G^\epsilon) - (\mu_t^\delta*G^\epsilon)\log(\mu_t^\delta*G^\epsilon)]
	\right|
	\lesssim_{\delta,\epsilon} ||\eta'||_{L^\infty} \left(
	\mu_0*G^\epsilon + \int_0^T \mu_t*G^\epsilon dt
	\right)\japangle{v}.
	\]
	Since $\mu$ has finite second order moments, this last expression belongs to $L_v^1$. By the Dominated Convergence Theorem,
	\[
	\frac{d}{dt}\mathcal{H}_\epsilon[\mu_t^\delta] 	= \int_{\mathbb{R}^d}\left[
	(\partial_t\mu_t^\delta) * G^\epsilon
	\right](\log (\mu_t^\delta*G^\epsilon) + 1) dv 	
	= \int_{\mathbb{R}^d}(\partial_t\mu_t^\delta) \cdot [G^\epsilon*\log(\mu_t^\delta*G^\epsilon)] dv
	\]
	The last line is achieved by the self-adjointness of convolution with $G^\epsilon$ and eliminating the constant term due to the conserved mass of $\mu^\delta$. Integrating in $t$, we obtain
	\begin{align}
		\label{eq:FTCreg}
		\begin{split}
			\mathcal{H}_\epsilon [\mu_r^\delta] - \mathcal{H}_\epsilon[\mu_s^\delta] 	&= \int_s^r \int_{\mathbb{R}^d} (\partial_t\mu_t^\delta) \cdot [G^\epsilon*\log(\mu_t^\delta*G^\epsilon)] dv dt	\\
			&=\frac{1}{2}\int_s^r\iint_{\mathbb{R}^{2d}}[\tn
			G^\epsilon *\log (\mu_t^\delta*G^\epsilon)]\cdot dM_t^\delta dt 	\\
			&= \frac{1}{2}\int_s^r\iiR \tn \frac{\delta \mathcal{H}_\epsilon}{\delta \mu_t^\delta}\cdot dM_t^\delta dt.
		\end{split}
	\end{align}  
	We now turn to establishing estimates independent of $\delta>0$ to pass to the limit.
	
	\textitunder{Estimates on the right-hand side of~\eqref{eq:FTCreg}:}
	\\
	According to Lemma~\ref{lem:esttnfirstvar}, we have the estimate
	\[
	\left|
	\tn \frac{\delta \mathcal{H}_\epsilon}{\delta \mu^\delta}
	\right| \lesssim_{\epsilon,E} |v|^p+|v_*|^p,
	\]
	where $p \leq 1$. By the first moment assumption of $M_t$, we have
	\begin{align*}
		&\quad \int_0^T\iiR\left|\tn \frac{\delta \mathcal{H}_\epsilon}{\delta \mu_t^\delta}\right| d|M_t|(v,v_*)dt \lesssim_{\epsilon,E} \int_0^T \iiR |v| + |v_*| d|M_t|(v,v_*)dt < \infty.
	\end{align*}
	This estimate also extends to $M_t^\delta$
	\[
	\int_0^T \iiR \left|
	\tn\firstvar{\mathcal{H}_\epsilon}{\mu_t^\delta}
	\right|d|M_t^\delta|(v,v_*)dt < \infty.
	\]
	Note that these estimates are independent of $\delta>0$.
	
	\textitunder{Convergence $\delta\to 0$:}
	\\
	Firstly, we establish the following identity which will be useful later. For fixed functions $f^1, \, f^2$ we have
\begin{align}
\label{eq:expandtndiff}
\begin{split}
&\quad \tn [G^\epsilon * f^1] - \tn [G^\epsilon * f^2] 	\\
&= |v-v_*|^{1+\frac{\gamma}{2}}\Pi[v-v_*](\nabla [G^\epsilon * f^1] - \nabla [G^\epsilon * f^2] - (\nabla_*[G^\epsilon * f^1]_* - \nabla_*[G^\epsilon * f^2]_*)) 	\\
&= |v-v_*|^{1+\frac{\gamma}{2}}\Pi[v-v_*]\int_{\R^d}(\nabla G^\epsilon(v - v') - \nabla G^\epsilon(v_* - v'))(f^1(v') - f^2(v'))dv'.
\end{split}
\end{align}	
	Using the weak in time continuity of $\mu$, we can consider
	\[
	|\mu_t^\delta * G^\epsilon(v') - \mu_t*G^\epsilon(v')|
	\leq \int_{-1}^1 \eta(t') |\langle\mu_{t-\delta t'}, G^\epsilon (v' - \cdot )\rangle - \langle\mu_{t}, G^\epsilon (v' - \cdot )\rangle|dt'.
	\]
	The $\cdot$ stands for the convoluted variable. Since $t$ belongs to a compact set, the function $t\mapsto \langle \mu_t, G^\epsilon (v'-\cdot)\rangle $ is \textit{uniformly} continuous from the weak continuity of $\mu$. In particular, using the continuity in $v'$ and the lower bound from Lemma~\ref{lem:carlen} we conclude that for any $R>0$
	\begin{equation}\label{eq:logconvergence}
		|\log(\mu_t^\delta * G^\epsilon)-\log(\mu_t * G^\epsilon)|\to 0\qquad\mbox{uniformly on $B_R$.}
	\end{equation}
	Therefore by Lemma~\ref{lem:carlen}, denoting $w = |v-v_*|^{1+\frac{\gamma}{2}}$, and using~\eqref{eq:expandtndiff} with $f^1 = \log (\mu_t^\delta * G^\epsilon)$ and $f^2 = \log (\mu_t *G^\epsilon)$, we have
	$$
	\begin{array}{l}
		\displaystyle\left|
		\tn \firstvar{\mathcal{H}_\epsilon}{\mu_t^\delta} - \tn \firstvar{\mathcal{H}_\epsilon}{\mu_t}
		\right|=|\tn G^{\epsilon}*\log (\mu_t^\delta * G^{\epsilon})(v,v_*)
		-
		\tn G^{\epsilon}*\log (\mu_t * G^{\epsilon})(v,v_*)|\\ 
		\displaystyle\le \int_{\R^d}w|\nabla G^{\epsilon}(v-v')-\nabla G^{\epsilon}(v_*-v')||\log (\mu_t^\delta * G^{\epsilon}(v'))-\log (\mu_t * G^{\epsilon}(v'))|\;dv'\\
		\displaystyle \le \int_{B_{R_0}^c}w|\nabla G^{\epsilon}(v-v')-\nabla G^{\epsilon}(v_*-v')|C_\epsilon\japangle{v'}\;dv'
		\\
		\displaystyle \qquad +\sup_{ B_{R_0}}|\log(\mu_t^\delta * G^{\epsilon})-\log(\mu_t * G^{\epsilon})|\int_{B_{R_0}} w |\nabla G^{\epsilon}(v-v')-\nabla G^{\epsilon}(v_*-v')|\;dv'.
	\end{array}
	$$
	For a fixed $(v,v_*)$, we obtain the convergence to zero by taking $\delta\to 0$ and $R_0\to \infty$ in the previous estimate. This holds for all $\gamma\in[-4,0]$ by taking advantage of the regularity of $G^\epsilon$. Using continuity, we obtain that for any $R>0$
	\begin{equation}\label{eq:uniformconvergence}
		\left|\tn \firstvar{\mathcal{H}_\epsilon}{\mu_t^\delta}(v,v_*)-\tn \firstvar{\mathcal{H}_\epsilon}{\mu_t}(v,v_*)\right|\to 0\qquad\mbox{uniformly on $[0,T]\times B_R\times B_R$}.
	\end{equation}
	We turn to the limit estimate for the right hand side of~\eqref{eq:FTCreg}. For any $R>0$, we have
	\begin{align*}
		&\quad \left|
		\int_s^r \iiR \tn \firstvar{\mathcal{H}_\epsilon}{\mu_t^\delta}\cdot dM_t^\delta dt - \int_s^r \iiR \tn \firstvar{\mathcal{H}_\epsilon}{\mu_t}\cdot dM_t dt
		\right| 	\\
		&\leq \left|
		\int_s^r \iiR \left(
		\tn \firstvar{\mathcal{H}_\epsilon}{\mu_t^\delta} - \tn \firstvar{\mathcal{H}_\epsilon}{\mu_t}
		\right)\cdot dM_t^\delta dt
		\right| + \left|
		\int_s^r\iiR \tn \firstvar{\mathcal{H}_\epsilon}{\mu_t}\cdot dM_t^\delta dt - \int_s^r \iiR \tn \firstvar{\mathcal{H}_\epsilon}{\mu_t}\cdot dM_t dt
		\right| 	\\
		&\leq \int_s^r \iint_{B_R \times B_R} \left|
		\tn \firstvar{\mathcal{H}_\epsilon}{\mu_t^\delta} - \tn \firstvar{\mathcal{H}_\epsilon}{\mu_t}
		\right|d|M_t^\delta| dt + \int_s^r \iint_{(B_R \times B_R)^C} \left|
		\tn \firstvar{\mathcal{H}_\epsilon}{\mu_t^\delta} - \tn \firstvar{\mathcal{H}_\epsilon}{\mu_t}
		\right|d|M_t^\delta| dt  + o(1).
	\end{align*}
	The last term is $o(1)$ as $\delta\to 0$ due to similar estimates from the previous step. By sending $\delta\to 0$ (the first term vanishes due to~\eqref{eq:uniformconvergence}) and then sending $R\to \infty$ (the second term vanishes again due to the estimate from the previous step), we obtain the convergence
	\begin{equation}\label{eq:convrhs}
		\lim_{\delta\to 0} \frac{1}{2}\int_s^r \iiR \tn \firstvar{\mathcal{H}_\epsilon}{\mu_t^\delta}\cdot dM_t^\delta dt = \frac{1}{2}\int_s^r \iiR \tn \firstvar{\mathcal{H}_\epsilon}{\mu_t}\cdot dM_t^\delta dt.
	\end{equation}

	\textitunder{Convergence of the left-hand side of~\eqref{eq:FTCreg}}
	\\
	By \eqref{eq:logconvergence}, Lemma~\ref{lem:carlen} and the uniform bound on the second moment, we have that
	\begin{align*}
		|\mathcal{H}_{\epsilon} [\mu_t^\delta] - \mathcal{H}_{\epsilon}[\mu_t]| 	&\leq \int_{\mathbb{R}^d} | (\mu_t^\delta*G^{\epsilon})\log (\mu_t^\delta *G^{\epsilon})(v) - (\mu_t*G^{\epsilon})\log (\mu_t*G^{\epsilon})(v)| dv 	\\
		&\to 0, \quad \text{as }\delta \to 0.
	\end{align*}
	Therefore, by the previous equation and \eqref{eq:convrhs} we can take $\delta \to 0$ in \eqref{eq:FTCreg} to obtain
	\[
	\mathcal{H}_{\epsilon} [\mu_r] - \mathcal{H}_{\epsilon}[\mu_s] = \frac{1}{2}\int_s^r\iint_{\mathbb{R}^{2d}} \tn\firstvar{\mathcal{H}_\epsilon}{\mu_t} \cdot dM_t(v,v_*) dt,
	\]
	which is the desired result.
\end{proof}
\section{JKO scheme for $\epsilon$-Landau equation}
\label{sec:JKO}
This section is devoted to the proof of Theorem~\ref{thm:existcms} after a series of preliminary lemmata. Our construction of curves of maximal slope in Theorem~\ref{thm:existcms} uses the basic minimizing movement/variational approximation scheme of Jordan et al.~\cite{jordan_variational_1998}. Fix a small time step $\tau > 0$ and initial datum $\mu_0\in\mathscr{P}_{2,E}(\Rd)$ and consider the recursive minimization procedure for $n\in\mathbb{N}$
\begin{equation}
	\label{eq:JKO1}
	\nu_0^\tau := \mu_0, \qquad \nu_n^\tau \in \mbox{argmin}_{\lambda\in\mathscr{P}_{2,E}}\left[
	\mathcal{H}_{\epsilon}(\lambda) + \frac{1}{2\tau}d_L^2(\nu_{n-1}^\tau,\lambda)
	\right].
\end{equation}
Then, we concatenate these minimizers into a curve by setting
\begin{equation}
	\label{eq:JKO2}
	\mu_0^\tau:= \mu_0, \qquad \mu_t^\tau := \nu_n^\tau, \, \mbox{ for }t\in((n-1)\tau,n\tau]. 
\end{equation}
The scheme given by~\eqref{eq:JKO1} and~\eqref{eq:JKO2} satisfies the abstract formulation in~\cite{ambrosio_gradient_2008-1} giving
\begin{proposition}
	[Landau JKO scheme]
	\label{prop:LJKO}
	For any $\tau>0$ and $\mu_0\in\mathscr{P}_{2,E}(\Rd)$, there exists $\nu_n^\tau\in\PtwoE$ for every $n\in\mathbb{N}$ as described in~\eqref{eq:JKO1}. Furthermore, up to a subsequence of $\mu_t^\tau$ described in~\eqref{eq:JKO2} as $\tau\to0$, there exists a locally absolutely continuous curve $(\mu_t)_{t\geq 0}$ such that
	\[
	\mu_t^\tau \rightharpoonup \mu_t, \qquad \forall t\in[0,\infty).
	\]
\end{proposition}
\begin{proof}
	Our metric setting is $(\mathscr{P}_{\mu_0},d_L)$ (see Theorem~\ref{them:existdL}) with the weak topology $\sigma$. This space is essentially $\PtwoE$ except we need to make sure that $d_L$ is a proper metric, hence we remove the probability measures with infinite Landau distance. We follow the proof of Erbar~\cite{erbar_gradient_2016} which consists in verifying~\cite[Assumptions 2.1 a,b,c]{ambrosio_gradient_2008-1}. These assumptions are listed and verified now.
	\begin{enumerate}
		\item \textbf{$\mathcal{H}_{\epsilon}$ is sequentially $\sigma$-lsc on $d_L$-bounded sets:}  Suppose $\mu_n\in\PtwoE \rightharpoonup \mu \in \PtwoE$, this implies $\mu_n*G^{\epsilon} \rightharpoonup \mu*G^{\epsilon}$ in $\mathscr{P}_2(\Rd)$. It is known that 
		\[
		\mathcal{H}(\mu) = \left\{
		\begin{array}{ll}
			\iR f(v) \log f(v) dv, 	&\mu = f \mathcal{L} 	\\
			+\infty, 	&\text{else}
		\end{array}
		\right.
		\]
		is $\sigma$-lsc and since $\mathcal{H}_{\epsilon}(\mu) = \mathcal{H}(\mu*G^{\epsilon})$, we achieve the first property.
		\item \textbf{$\mathcal{H}_{\epsilon}$ is lower bounded:} By Carlen-Carvalho Lemma~\ref{lem:carlen} for fixed $\epsilon>0$, $\log (\mu*G^{\epsilon})$ is uniformly lower bounded by a linearly growing term. For fixed $\mu\in\PtwoE$, we have, with Cauchy-Schwarz
		\[
		\mathcal{H}_{\epsilon}(\mu) \gtrsim_\epsilon -\iR \langle v \rangle \mu*G^{\epsilon}(v) dv \geq -\left(
		\iR \langle v \rangle^2 \mu*G^{\epsilon}(v) dv
		\right)^\frac{1}{2} \geq -(\mathcal{O}(\epsilon)+E)^\frac{1}{2} > -\infty.
		\]
		\item \textbf{$d_L$-bounded sets are relatively sequentially $\sigma$-compact:} This is one of the consequences from Theorem~\ref{them:existdL}.
	\end{enumerate}
	The existence of minimizers, $\nu_n^\tau$, to~\eqref{eq:JKO1} and limits, $\mu_t$, to~\eqref{eq:JKO2} is guaranteed from~\cite[Corollary 2.2.2]{ambrosio_gradient_2008-1} and~\cite[Proposition 2.2.3]{ambrosio_gradient_2008-1}, respectively.
\end{proof}
At the abstract level, the limit curve constructed in Proposition~\ref{prop:LJKO} has no relation to $\sqrt{D_\epsilon}$. The following lemmata bridge this gap. 
\begin{lemma}
	\label{lem:Depsslope}
	For any $\mu_0\in \mathscr{P}_2(\Rd)$, we have
	\[
	\sqrt{D_\epsilon(\mu_0)} \leq |\partial^-\mathcal{H}_{\epsilon}|(\mu_0).
	\]
\end{lemma}
\begin{proof}
	For fixed $\epsilon,R_1,R_2>0$ and $\gamma\in\R$, take $T>0$ from Theorem~\ref{thm:auxpde} in Appendix~\ref{sec:auxpde} and the unique weak solution $\mu \in C([0,T];\mathscr{P}_2(\Rd))$ to
	\begin{equation*}
		\left\{
		\begin{array}{rcl}
			\partial_t \mu 	&= 	&\nabla\cdot \{
			\mu\phi_{R_1} \iR  \phi_{R_1*} \psi_{R_2}(v-v_*)|v-v_*|^{\gamma+2} \Pi[v-v_*](J_0^\epsilon - J_{0*}^\epsilon)d\mu(v_*)
			\} 	\\
			\mu(0) 	&= 	&\mu_0
		\end{array}
		\right..
	\end{equation*}
	The functions  $0 \leq \phi_{R_1}, \psi_{R_2} \leq 1$ are smooth cut-off functions with the following properties
	\[
	\phi_{R_1}(v) = \left\{
	\begin{array}{cl}
		1, 	&|v| \leq R_1 	\\
		0, 	&|v| \geq R_1+1
	\end{array}
	\right., 	\quad
	\psi_{R_2}(z) = \left\{
	\begin{array}{cl}
		0, 	&|z| \leq 1/R_2 	\\
		1, 	&|z| \geq 2/R_2
	\end{array}
	\right..
	\]
	The notation $J_0^\epsilon$ from~\Cref{sec:auxpde} means
	\[
	J_0^\epsilon = \nabla G^{\epsilon} * \log [\mu_0 * G^{\epsilon}]\in C^\infty(\Rd;\Rd).
	\]
	For this proof alone, we define the reduced $\epsilon$-entropy-dissipation
	\[
	D_{\epsilon}^{R_1,R_2}(\mu_0) := \frac{1}{2}\iiR \phi_{R_1}\phi_{R_1*} \psi_{R_2}(v-v_*)|v-v_*|^{\gamma+2}\left|
	\Pi[v-v_*](J_0^\epsilon - J_{0*}^\epsilon)
	\right|^2d\mu_0(v)d\mu_0(v_*).
	\]
	On the other hand, as the $\epsilon$-entropy dissipation comes from the negative time derivative of entropy, we have
	\begin{align*}
		&\quad D_{\epsilon}^{R_1,R_2}(\mu_0) = \lim_{t\downarrow 0} \frac{\mathcal{H}_{\epsilon}(\mu_0) - \mathcal{H}_{\epsilon}(\mu_t)}{t} = \lim_{t\downarrow 0} \frac{\mathcal{H}_{\epsilon}(\mu_0) - \mathcal{H}_{\epsilon}(\mu_t)}{d_L(\mu_0,\mu_t)}\frac{d_L(\mu_0,\mu_t)}{t} 	\\
		&\leq \lim_{t\downarrow 0}\left\{ \frac{\mathcal{H}_{\epsilon}(\mu_0) - \mathcal{H}_{\epsilon}(\mu_t)}{d_L(\mu_0,\mu_t)} \right. \times \frac{1}{t}	\\
		&\quad \left.\times \left(
		\int_0^t\sqrt{\frac{1}{2}\iiR \phi_{R_1}^2\phi_{R_1*}^2\psi_{R_2}^2|v-v_*|^{\gamma+2}|\Pi[v-v_*](J_0^\epsilon - J_{0*}^\epsilon)|^2}d\mu_s(v)d\mu_s(v_*)ds
		\right)\right\} 	\\
		&\leq |\partial\mathcal{H}_{\epsilon}|(\mu_0)\sqrt{D_\epsilon^{R_1,R_2}(\mu_0)}.
	\end{align*}
In the first inequality, we estimated $d_L(\mu_0,\mu_t)$ by considering the PDE in this lemma as the grazing collision equation with $M = -(\mu\otimes \mu) \tn \log \mu_0$.
In the last inequality, we have used the Lebesgue differentiation theorem with strong-weak convergence since $\mu$ is continuous in time as well as the fact that $\phi_{R_1}^2 \leq \phi_{R_1}$ and $\psi_{R_2}^2 \leq \psi_{R_2}$ since $0 \leq \phi_{R_1},\psi_{R_2}\leq 1$. We are left with the inequality
\[
\sqrt{D_\epsilon^{R_1,R_2}(\mu_0)} \leq |\partial\mathcal{H}_{\epsilon}|(\mu_0), \quad \forall R_1,R_2>0.
\]
Owing to the many regularisations applied, the $\epsilon$-entropy-dissipation $\mu \mapsto D_\epsilon^{R_1, \, R_2}(\mu)$ is continuous with respect to weak convergence of probability measures. By considering weakly convergent sequences and passing to the limit inferior, we deduce the same inequality with the relaxed slope
\[
\sqrt{D_\epsilon^{R_1,R_2}(\mu_0)} \leq |\partial^-\mathcal{H}_{\epsilon}|(\mu_0), \quad \forall R_1,R_2>0.
\]
	As functions of $R_1,R_2$ individually, $D_\epsilon^{R_1,R_2}(\mu_0)$ is non-decreasing. Furthermore, the integrand of $D_\epsilon^{R_1,R_2}(\mu_0)$ converges to the integrand of $D_\epsilon(\mu_0)$ pointwise $\mu_0$-almost every $v,v_*$. Thus, an application of the monotone convergence theorem in the limit $R_1,R_2\to \infty$ on the above inequality completes the proof.
\end{proof}
\begin{lemma}
	\label{lem:slopeissug}
	$|\partial^-\mathcal{H}_{\epsilon}|$ is a strong upper gradient for $\mathcal{H}_{\epsilon}$ in $\mathscr{P}_{\mu_0}(\Rd)$ where $\mu_0 \in \PtwoE$.
\end{lemma}
\begin{proof}
	Fix $\lambda,\nu\in\mathscr{P}_{\mu_0}(\Rd)$ so that by the triangle inequality of Theorem~\ref{them:existdL}, we have $d_L(\lambda,\nu)<\infty$. Now by Proposition~\ref{prop:mincurve}, there exists a pair of curves $(\mu,M)\in\mathcal{GCE}_1^E$ connecting $\lambda,\nu$ and $\mathcal{A}(\mu_t,M_t) = d_L^2(\lambda,\nu)$ for almost every $t\in[0,1]$. Using Remark~\ref{rk:stronguppergradient} and Lemma~\ref{lem:Depsslope}, we have
	\[
	|\mathcal{H}_{\epsilon}(\lambda) - \mathcal{H}_{\epsilon}(\nu)| \leq \int_0^1 \sqrt{D_\epsilon(\mu_t)}|\dot{\mu}|(t)dt \leq \int_0^1 |\partial^-\mathcal{H}_{\epsilon}|(\mu_t)|\dot{\mu}|(t)dt.
	\]
\end{proof}
We now have all the ingredients to prove Theorem~\ref{thm:existcms} so that we can relate curves of maximal slope to weak solutions of the $\epsilon$-Landau equation.
\begin{proof}
	[Proof of Theorem~\ref{thm:existcms}]
	Take a limit curve $\mu_t$ constructed in Proposition~\ref{prop:LJKO}. By the previous Lemma~\ref{lem:slopeissug}, the assumptions of~\cite[Theorem 2.3.3]{ambrosio_gradient_2008-1} are fulfilled so the curve is of maximal slope with respect to $|\partial^-\mathcal{H}_{\epsilon}|$ and satisfies the associated energy dissipation inequality
	\[
	\mathcal{H}_{\epsilon}(\mu_r) - \mathcal{H}_{\epsilon}(\mu_s) + \frac{1}{2} \int_s^r |\partial^-\mathcal{H}_{\epsilon}(\mu_t)|^2 dt + \frac{1}{2} \int_s^r |\dot{\mu}|^2(t) dt \leq 0.
	\]
	The inequality of Lemma~\ref{lem:Depsslope} gives
	\[
	\mathcal{H}_{\epsilon}(\mu_r) - \mathcal{H}_{\epsilon}(\mu_s) + \frac{1}{2} \int_s^r D_\epsilon(\mu_t) dt + \frac{1}{2} \int_s^r |\dot{\mu}|^2(t) dt \leq 0,
	\]
	which is precisely the statement that the limit curve $\mu_t$ is a curve of maximal slope with respect to $\sqrt{D_\epsilon}$.
\end{proof}
\begin{remark}
	\label{rk:JKOremarks}
	The results of Proposition~\ref{prop:LJKO} and Lemma~\ref{lem:Depsslope} can be generalized to other regularization kernels $G^{s,\epsilon}$, in particular, the Maxwellian regularization. However, this is not the case for Lemma~\ref{lem:slopeissug} since the proof relies on Proposition~\ref{prop:strongupgradeps}, see Remark~\ref{rk:Glogdiff}.
\end{remark}

\section{Recovering the full Landau equation as $\epsilon \to 0$}
\label{sec:epsto0}
Theorems~\ref{thm:epsequiv} and~\ref{thm:existcms} provide the basic existence theory for the $\epsilon>0$ approximation of the Landau equation. In this section, we prove the $\epsilon\downarrow0$ analogue of Theorem~\ref{thm:epsequiv} which is Theorem~\ref{thm:fullequiv}. By definition, both H-solutions and curves of maximal slope to the full Landau equation dissipate the entropy. Therefore, the assumption of finite initial entropy~\ref{itm:ass2} automatically ensures
\[
\sup_{t\in[0,T]}\mathcal{H}[f_t] = \sup_{t\in[0,T]}\int_{\R^3}f_t \log f_t < +\infty.
\]
In the sequel, every quotation of~\ref{itm:ass2} will refer to this bound.
\begin{proof}[Sketch of the proof of Theorem~\ref{thm:fullequiv}]
	By repeating the proof of Theorem~\ref{thm:epsequiv}, we see that the crucial ingredient is the chain rule~\eqref{eq:FTCentropy} in Proposition~\ref{prop:strongupgradeps}. For now assume the following
	\begin{claim}
		\label{claim:fullchainrule}
		Assume~\ref{itm:ass1}, ~\ref{itm:ass2}, ~\ref{itm:ass3} and let $M$ be any grazing rate such that $(\mu,M) \in \mathcal{GCE}_T^E$ and
		\[
		\int_0^T \mathcal{A}(\mu_t,M_t) dt < \infty.
		\]
		Then we have the chain rule
		\begin{equation}
			\label{eq:fullchainrule}
			\mathcal{H}[\mu_r] - \mathcal{H}[\mu_s] = \frac{1}{2} \int_s^r \iiRsix \tn \left[
			\firstvar{\mathcal{H}}{\mu}
			\right]\cdot dM_t dt.
		\end{equation}
	\end{claim}
	By following the steps of the proof of Theorem~\ref{thm:epsequiv} and using~\eqref{eq:fullchainrule} instead of~\eqref{eq:FTCentropy}, one completes the proof of Theorem~\ref{thm:fullequiv}. We dedicate this section to proving Claim~\ref{claim:fullchainrule}. 
	
	Equation~\eqref{eq:fullchainrule} is clearly the $\epsilon\downarrow 0$ limit of~\eqref{eq:FTCentropy}. The left-hand side of~\eqref{eq:fullchainrule} can be obtained from the left-hand side of~\eqref{eq:FTCentropy} using the finite entropy assumption~\ref{itm:ass2} and the fact that $\epsilon \mapsto \mathcal{H}_\epsilon[\mu_t]$ is non-increasing for every $t$. We refer to~\cite[Proof of Proposition 4.2; Step 4: part d)]{erbar_gradient_2016} for more details on a similar argument.
	
	The difficulty remains in deducing that the right-hand side of~\eqref{eq:FTCentropy} converges to the right-hand side of~\eqref{eq:fullchainrule} as $\epsilon\downarrow 0$ given by
	\begin{equation}
		\label{eq:rhschainconv}
		\int_0^T \iiRsix \tn \firstvar{\mathcal{H}_\epsilon}{\mu}\cdot dM_t dt \to \int_0^T \iiRsix \tn \firstvar{\mathcal{H}}{\mu}\cdot dM_t dt, \quad \epsilon \downarrow 0
	\end{equation}
	under the additional assumptions~\ref{itm:ass1}, ~\ref{itm:ass2}, ~\ref{itm:ass3} on $f$. The key result which we will use repeatedly in this section is the following theorem which is a specific case of the result in~\cite[Chapter 4, Theorem 17]{R88}.
	\begin{theorem}[Extended Dominated Convergence Theorem (EDCT)]
		\label{thm:EDCT}
		Let $(H_\epsilon)_{\epsilon>0}$ and $(I_\epsilon)_{\epsilon>0}$ be sequences of  measurable functions on $X$ satisfying $I_\epsilon \geq 0$ and suppose there exists measurable functions $H,\;I$ satisfying
		\begin{enumerate}
			\item \label{item:edctmajor} $|H_\epsilon| \leq I_\epsilon$ for every $\epsilon>0$ and pointwise a.e.
			\item \label{item:edctptconv} $H_\epsilon$ and $I_\epsilon$ converge pointwise a.e. to $H$ and $I$, respectively.
			\item \label{item:edctintconv} \[
			\lim_{\epsilon\downarrow0}\int_X I_\epsilon = \int_X I <\infty.
			\]
		\end{enumerate}
		Then, we have the convergence
		\[
		\lim_{\epsilon\downarrow0} \int_X H_\epsilon = \int_X H.
		\]
	\end{theorem}
	Setting $M = m \mathcal{L}\otimes \mathcal{L}$ (valid by Proposition~\ref{lem:actfun}) and using Young's inequality on the right-hand side of~\eqref{eq:FTCentropy}, we obtain the majorants
	\[
	\tn \left[
	\frac{\delta \mathcal{H}_{\epsilon}}{\delta \mu}
	\right]\cdot m_t \leq \frac{1}{2}ff_* \left|
	\tn \left[
	\frac{\delta \mathcal{H}_{\epsilon}}{\delta \mu}
	\right]
	\right|^2 + \frac{1}{2}\frac{|m_t|^2}{ff_*}. 
	\]
	Notice that the first term is precisely the integrand of $D_\epsilon$ while the second term is the integrand of the action functional $\mathcal{A}(\mu_t,M_t)$ which has no dependence on $\epsilon$ and is henceforth ignored. We can apply EDCT~\ref{thm:EDCT} with $X = (0,T) \times \mathbb{R}^6$ to prove~\eqref{eq:rhschainconv} once we show
	\begin{equation}
		\label{eq:dissconv}
		\int_0^T \iiRsix ff_* \left|
		\tn \left[
		\frac{\delta \mathcal{H}_{\epsilon}}{\delta \mu}
		\right]
		\right|^2dv_*dvdt \to \int_0^T \iiRsix ff_* \left|
		\tn \left[
		\frac{\delta \mathcal{H}}{\delta \mu}
		\right]
		\right|^2dv_*dvdt, \quad \epsilon\downarrow 0.
	\end{equation}
	The pointwise a.e. convergence hypothesis of EDCT~\ref{thm:EDCT} is straightforward based on the regularization of $\mathcal{H}_\epsilon$ through $G^\epsilon$. Focusing on~\eqref{eq:dissconv}, we will use a standard Dominated Convergence Theorem (DCT) for the integration in the $t$ variable, by proving
	\begin{align}
		\label{eq:tedct}
		\begin{split}
			\iiRsix \frac{1}{2}ff_* \left|
			\tn \left[
			\frac{\delta \mathcal{H}_{\epsilon}}{\delta \mu}
			\right]
			\right|^2 dv_*dv &\to \iiRsix \frac{1}{2}ff_* \left|
			\tn \left[
			\frac{\delta \mathcal{H}}{\delta \mu}
			\right]
			\right|^2dv_*dv, \quad \text{a.e. }t,     \\
			\iiRsix \frac{1}{2}ff_* \left|
			\tn \left[
			\frac{\delta \mathcal{H}_{\epsilon}}{\delta \mu}
			\right]
			\right|^2 dv_*dv    &\leq C\iiRsix \frac{1}{2}ff_* \left|
			\tn \left[
			\frac{\delta \mathcal{H}}{\delta \mu}
			\right]
			\right|^2dv_*dv, \quad \text{a.e. }t \quad \forall \epsilon>0,
		\end{split}
	\end{align}
	where $C>0$ is a constant independent of $\epsilon>0$. The estimate of~\eqref{eq:tedct} guarantees the $L_t^1$ majorisation due to the finite entropy-dissipation assumption~\ref{itm:ass3}.
\end{proof}
Our estimates in this section accomplish both the convergence and the estimate of~\eqref{eq:tedct} by nested application of EDCT~\ref{thm:EDCT}. The significance of all three assumptions~\ref{itm:ass1}, ~\ref{itm:ass2}, and \ref{itm:ass3} will be apparent in proving the convergence in~\eqref{eq:tedct}. 

\begin{remark}
	In this section, the only properties of $G^\epsilon$ we use are that it is a non-negative radial approximate identity with sufficiently many moments. As in the construction of minimizing movement curves in Section~\ref{sec:JKO}, the results of this section can be achieved with other radial approximate identities.
\end{remark}
\subsection{Outline of technical strategy to prove~\eqref{eq:tedct}}
\label{sec:outlinestrat}

The need to apply EDCT~\ref{thm:EDCT} instead of the more classical Lebesgue DCT is that we are unable to prove pointwise estimates in $v$ for the function $v\to f\iRthree f_* \left| \tn \left[\firstvar{\mathcal{H}_\epsilon}{f}\right]\right|^2 dv_*$. Instead, our estimates in this section rely on the self-adjointness of convolution against radial exponentials (SACRE) to construct a convergent majorant in $\epsilon$.

\textitunder{Step 1: Finding majorants and appealing to EDCT~\ref{thm:EDCT}}
\\
We seek to find pointwise a.e. majorants in the $v$ variable
\[
f\iRthree f_* \left|
\tn \left[
\frac{\delta \mathcal{H}_{\epsilon}}{\delta \mu}
\right]
\right|^2dv_* \leq I_\epsilon^{1}(v),
\]
where $I_\epsilon^1(v)$ satisfies the hypothesis for the majorant in EDCT~\ref{thm:EDCT}. We show that $I_\epsilon^1$ converges pointwise to some $I^1$, since $I_\epsilon^1$ depends on $\epsilon$ only through convolutions against $G^{\epsilon}$, which is an approximation of the identity. Hence, we are left with showing the integral convergence Item~\ref{item:edctintconv} of EDCT~\ref{thm:EDCT}
\[
\iRthree I_\epsilon^1(v) dvdt \to \iRthree I^1(v)dv, \quad \epsilon \to 0.
\]
\\
\textitunder{Step 2: Use SACRE with $G^\epsilon$}
\\
To show the integral convergence for $I^1_\epsilon$, we find functions $A^1$ and $B^1$ such that 
\[
I_\epsilon^1(v) \leq A^1(v) (G^{\epsilon}* B^1)(v)
\]
and apply EDCT~\ref{thm:EDCT}. As in the previous step, the pointwise convergence is easily proved. Hence, we are left to show the integral convergence
\[
\iRthree A^1 (G^{\epsilon}*B^1) dv \to \iRthree A^1 B^1, \quad \epsilon\to 0.
\] 
The key observation is applying SACRE to obtain
\[
\iRthree A^1 (G^{\epsilon}*B^1) = \iRthree \underbrace{(G^{\epsilon}*A^1)B^1}_{=:I_\epsilon^2}.
\]
Therefore, we have reduced the problem to showing integral convergence Item~\ref{item:edctintconv} of EDCT for $I_\epsilon^2$ (as the pointwise convergence is easily proved).

\textitunder{Step 3: Reiterate step 2}
\\
We repeat the process outlined in Step 2 by finding functions $A^2$ and $B^2$ such that we have the pointwise bound
\[
I_\epsilon^2(v) \leq A^2(v) (G^{\epsilon}*B^2)(v).
\]
Again the pointwise convergence for the majorant follows easily, hence we only need to check the integral convergence Item~\ref{item:edctintconv} of EDCT~\ref{thm:EDCT} given by
\[
\iRthree A^2 (G^{\epsilon}*B^2) \to \iRthree A^2 B^2.
\]
Using SACRE, we study instead the integral convergence of
$$
I_\epsilon^3(v) = (G^{\epsilon}*A^2)B^2.
$$
Eventually, after a finite number of times of finding majorants and applying SACRE, we will obtain a majorant $I_\epsilon^i$ for which the estimates and the convergence as $\epsilon\to0$ follows from the standard Lebesgue DCT, using the bound of the weighted Fisher information in terms of the entropy-dissipation (see Theorem~\ref{thm:desvillettes16}) and assumption~\ref{itm:ass3}.
\subsection{Preparatory results}
\label{sec:prep}
As mentioned in the previous section, for the final step of the proof we need a bound on the weighted Fisher information and a closely related variant in terms of the entropy-dissipation originally discovered by the third author in \cite{desvillettes_exponential_2006}.
\begin{theorem}
	\label{thm:desvillettes16}
	Suppose $\gamma\in(-4,0]$ and let $f\geq 0$ be a probability density belong to $L_{2-\gamma}^1\cap L \log L (\Rthree)$. We have 
	\[
	\iRthree f(v) \langle v\rangle^\gamma \left|\nabla \frac{\delta \mathcal{H}}{\delta f}\right|^2 dv + \iRthree f(v) \japangle{v}^\gamma \left|v\times
	\nabla \frac{\delta \mathcal{H}}{\delta f}
	\right|^2dv \leq C (1+ D_{w,\mathcal{H}}(f)),
	\]
	where $C>0$ is a constant depending only on the bounds of $m_{2-\gamma}(f)$ and the Boltzmann entropy, $\mathcal{H}[f]$, of $f$.
\end{theorem}
The estimate in this precise form can be found in \cite[Proposition 4, p. 10]{LD_PX}. We will refer to the second term on the left-hand side as a `cross Fisher information'. We mention here that assumption~\ref{itm:ass2} enters in the sequel since the constant $C>0$ in~\Cref{thm:desvillettes16} depends on bounds for $\mathcal{H}[f]$.

To decompose the entropy-dissipation in a manageable way that makes the cross Fisher term more apparent, we have the following linear algebra fact.
\begin{lemma}
	\label{lem:linalgcross}
	For $x,y\in\Rthree$, we have
	\[
	|x|^2 (y\cdot \Pi[x]y) = |x\times y|^2
	\]
\end{lemma}
\begin{proof}
	Without loss of generality, we assume neither $x,y=0$ or else the statement holds trivially. Let $\theta$ be an oriented angle between $x$ and $y$. We expand the definition of $\Pi[x]$ and observe
	\begin{align*}
		&|x|^2 (y\cdot \Pi[x] y) = y\cdot (|x|^2I - x\otimes x)y = |x|^2 |y|^2 - |x\cdot y|^2 = |x|^2|y|^2(1-\cos^2\theta) 	
		= |x|^2|y|^2\sin^2\theta 	\\ &= |x\times y|^2.
	\end{align*}
\end{proof}
The following lemma shows how we use assumption~\ref{itm:ass1} to control the singularity of the weight.
\begin{lemma}
	\label{lem:Jf}
	Given $\gamma\in(-3,0]$, assume that $f$ satisfies~\ref{itm:ass1} for some $0<\eta\le \gamma+3$, then we have for a.e. $t$
	\begin{equation}
		\label{eq:Jf}
		\iRthree f_*(t) |v-v_*|^\gamma dv_* \leq C_1(t)\japangle{v}^\gamma, \quad \iRthree f_*(t)|v_*|^2 |v-v_*|^\gamma dv_* \leq C_2(t) \japangle{v}^\gamma,
	\end{equation}
	where
	\[
	\begin{array}{rcl}
		||C_1||_{L^\infty(0,T)}&\lesssim_{\gamma,\eta}& ||\japangle{\cdot}^{-\gamma}f(t)||_{L^\infty\left(0,T;L^1\cap L^\frac{3-\eta}{3+\gamma-\eta}(\Rthree)\right)}\\ ||C_2||_{L^\infty(0,T)}&\lesssim_{\gamma,\eta}&||\japangle{\cdot}^{2-\gamma}f(t)||_{L^\infty\left(0,T;L^1\cap L^\frac{3-\eta}{3+\gamma-\eta}(\Rthree)\right)}.
	\end{array}
	\]
\end{lemma}
\begin{proof}
	We will only prove the first inequality of~\eqref{eq:Jf} since the second inequality uses the same procedure. We split the estimation for local $|v|\leq 1$ and far-field $|v|\geq 1$.

	\noindent\textitunder{$|v|\leq 1$}
	\\
	We split the integral over $v_*$ into two regions
	\begin{align*}
		\iRthree
		f_* |v-v_*|^\gamma dv_* &= \int_{|v-v_*|\geq 1} f_* |v-v_*|^\gamma dv_* + \int_{|v-v_*|\leq 1} f_* |v-v_*|^\gamma dv_* 	\\
		&\leq 1 + \int_{|v-v_*|\leq 1} f_* |v-v_*|^\gamma dv_* ,
	\end{align*}
	where we have used that $\iRthree f = 1$ and $\gamma\leq 0$. For the integral with the singularity, we apply Young's convolution inequality with conjugate exponents $\left(
	\frac{3-\eta}{3+\gamma-\eta}, \frac{-3+\eta}{\gamma}
	\right)$
	\[
	\int_{|v-v_*|\leq 1} f_* |v-v_*|^\gamma dv_* \leq || f * (\chi_{B_1}|\cdot|^\gamma)||_{L^\infty} \leq ||f||_{L^\frac{3-\eta}{3+\gamma-\eta}} ||\chi_{B_1} |\cdot |^\gamma||_{L^{\frac{-3+\eta}{\gamma}}}\leq \left(\frac{\omega_2}{\eta}\right)^{\frac{-3+\eta}{\gamma}}||f||_{L^\frac{3-\eta}{3+\gamma-\eta}}.
	\]
	Here, $\omega_2$ is the volume of the unit sphere in $\Rthree$. 
	\\
	\textitunder{$|v|\geq 1$}
	\\
	Once again, we split the integral into two parts
	\begin{align*}
		\iRthree f_* |v-v_*|^\gamma dv_* &= \int_{|v_*|\leq \frac{1}{2}|v|} f_* |v-v_*|^\gamma dv_* + \int_{|v_*|\geq \frac{1}{2}|v|} f_* |v-v_*|^\gamma dv_* 	\\
		&\leq 2^{-\gamma} |v|^\gamma\int_{|v_*|\leq \frac{1}{2}|v|} f_*  dv_* + 2^{-\gamma}|v|^\gamma\int_{|v_*|\geq \frac{1}{2}|v|} f_* |v_*|^{-\gamma}|v-v_*|^\gamma dv_*.
	\end{align*}
	The first term and second term come from the following inequalities based on their respective integration regions
	\[
	|v-v_*| \geq |v| - |v_*| \geq \frac{1}{2}|v|, \quad 1 \leq 2^{-\gamma}|v|^\gamma |v_*|^{-\gamma}.
	\] 
	We estimate the first integral using the unit mass of $f$, while the second integral is more delicate but again uses the splitting of the previous step to obtain
	\[
	\iRthree f_* |v-v_*|^\gamma dv_* \leq 2^{-\gamma}|v|^\gamma + 2^{-\gamma}|v|^\gamma\left(
	\int_{|v-v_*|\geq 1} f_* |v_*|^{-\gamma}|v-v_*|^\gamma dv_* + \int_{|v-v_*| \leq 1} f_* |v_*|^{-\gamma}|v-v_*|^\gamma dv_*
	\right).
	\]
	In the large brackets, the first integral can be estimated by $m_{-\gamma}(f)$. Now we use the same Young's inequality argument for the remaining integral to obtain
	\begin{equation*}
		\iRthree f_* |v-v_*|^\gamma dv_* \leq 2^{-\gamma}|v|^\gamma + 2^{-\gamma}|v|^\gamma\left(
		m_{-\gamma}(f) + \left(\frac{\omega_2}{\eta}\right)^{\frac{-3+\eta}{\gamma}}|| |\cdot |^{-\gamma} f||_{L^\frac{3-\eta}{3+\gamma-\eta}(\Rthree)}
		\right).
	\end{equation*}
	The proof is complete by combining the estimates for $|v|\leq 1$ and $|v|\geq 1$.
\end{proof}
\begin{lemma}[Peetre]
	\label{lem:peetre}
	For any $p\in\mathbb{R}$ and $x,y\in\Rd$, we have
	\[
	\frac{\langle x\rangle^p}{\langle y\rangle^p} \leq 2^{|p|/2} \langle x-y\rangle^{|p|}.
	\]
\end{lemma}
\begin{proof}
	Our proof follows~\cite{BN73}. Starting with the case $p=2$, for fixed vectors $a,b\in\Rd$ we have, with the help of Young's inequality,
	\begin{align*}
		&\quad 1+|a-b|^2 \leq 1 + |a|^2 + 2|a||b| + |b|^2 \leq 1 + 2|a|^2 + 2|b|^2    \\
		&\leq 2 + 2|a|^2 + 2|a|^2|b|^2 + 2|b|^2 = 2(1+|a|^2)(1+|b|^2).
	\end{align*}
	Dividing by $\japangle{b}^2$ and setting $a = x-y,b=-y$, we obtain the inequality for $p=2$
	\[
	\frac{\japangle{x}^2}{\japangle{y}^2} \leq 2\japangle{x-y}^2.
	\]
	By taking non-negative powers, this proves the inequality for $p\geq 0$. On the other hand, when we divided by $\japangle{b}^2$ we could have also set $a = x-y,\; b = x$ to obtain
	\[
	\frac{\japangle{y}^2}{\japangle{x}^2} \leq 2\japangle{x-y}^2.
	\]
	Taking strictly non-negative powers here proves the inequality for $p < 0$.
\end{proof}

Next, we prove an estimate for algebraic functions (growing or decaying) convoluted against $G^\epsilon$ with respect to the original function.
\begin{lemma}
	\label{lem:mombeps}
	For any $p\in \mathbb{R}$, we have
	\[
	\iR \japangle{w}^p G^\epsilon(v-w)dw \leq C \japangle{v}^p,
	\]
	where $C>0$ is a constant depending only on $|p|$ and $m_{|p|}(G)$.
\end{lemma}
\begin{proof}
	
	We use Peetre's inequality in Lemma~\ref{lem:peetre} to introduce $v-w$ into the angle brackets
	\begin{align*}
		&\quad \iR \japangle{w}^p G^\epsilon(v-w) dw \leq 2^{|p|/2}\japangle{v}^p \iR \japangle{v-w}^{|p|}G^\epsilon(v-w)dw    \\ 
		&= 2^{|p|/2}\japangle{v}^p \iR (1+|w|^2)^\frac{|p|}{2}\epsilon^{-d}G(w/\epsilon) dw = 2^{|p|/2}\japangle{v}^p \iR (1+\epsilon^2|w|^2)^\frac{|p|}{2}G(w) dw     \\
		&\le C_{|p|} \japangle{v}^p \left[
		1 + \epsilon^{|p|}\iR |w|^{|p|}G(w)dw
		\right] \le C_{|p|} \left[1+\epsilon^{|p|}m_{|p|}(G)\right]\japangle{v}^p
	\end{align*}
\end{proof}
We stress that Peetre's inequality~\ref{lem:peetre} is necessary for the estimate of Lemma~\ref{lem:mombeps} with \textit{non-positive} powers $p$ which we apply in the sequel.
Finally, the last result we will need is an integration by parts formula for the differential operator associated to the cross Fisher information.
\begin{lemma}[Twisted integration by parts]
	\label{lem:crossibp}
	Let $f,g$ be smooth scalar functions of $\Rthree$ which are sufficiently integrable. Then, we have the formula
	\begin{equation*}
		\iRthree (v\times \nabla_v g(v))f(v) dv = - \iRthree g(v) (v\times \nabla_v f(v))dv.
	\end{equation*}
	Here, the meaning of $v\times \nabla_v$ is
	\[
	v\times \nabla_v f(v) = (v^2\partial^3 f(v) - v^3 \partial^2 f(v), v^3 \partial^1 f(v) - v^1 \partial^3 f(v), v^1 \partial^2 f(v) - v^2 \partial^1 f(v)).
	\]
\end{lemma}
\subsection{Proof of~\eqref{eq:tedct} using EDCT~\ref{thm:EDCT}}
\label{sec:firstreduce}
We start by decomposing and estimating the integrand of $D_\epsilon$. With the help of Lemma~\ref{lem:linalgcross}, we expand the square term of the integrand to see
\begin{align*}
	\begin{split}
		\left|
		\tn \left[
		\frac{\delta \mathcal{H}_{\epsilon}}{\delta \mu}
		\right]
		\right|^2 =& \,|v-v_*|^{2+\gamma} |\Pi[v-v_*](b^\epsilon*a^\epsilon - b^\epsilon*a_*^\epsilon)|^2 	\\
		\leq &\,|v-v_*|^\gamma (4|v\times (b^\epsilon* a^\epsilon)|^2 + 4|v_*\times (b^\epsilon*a_*^\epsilon)|^2 
		\\
		&+ 4|v\times (b^\epsilon*a_*^\epsilon)|^2 + 4|v_*\times (b^\epsilon*a^\epsilon)|^2) 	\\
		\leq &\,4|v-v_*|^\gamma \underbrace{|v\times (b^\epsilon*a^\epsilon)|^2}_{\textcircled{1}} + 4|v-v_*|^\gamma \underbrace{|v_*\times (b^\epsilon*a_*^\epsilon)|^2}_{\textcircled{2}}    \\
		&+ 4|v|^2|v-v_*|^\gamma \underbrace{|b^\epsilon*a_*^\epsilon|^2}_{\textcircled{3}} + 4 |v_*|^2|v-v_*|^\gamma \underbrace{|b^\epsilon*a^\epsilon|^2}_{\textcircled{4}},
	\end{split}
\end{align*}
where we use the shorthand notation
\begin{equation}\label{eq:aandb}
	b^\epsilon=G^\epsilon\qquad\mbox{and}\qquad a^\epsilon=\nabla\log(G^\epsilon *f).    
\end{equation}
By using that $G^\epsilon$ is an approximation of the identity, we know that the integrand of $D_\epsilon$ converges pointwise a.e. to the integrand of $D$ as $\epsilon\downarrow0$. As well, each $\textcircled{i}$ for $i=1,2,3,4$ converge pointwise a.e. to 
\[
\textcircled{1} \to\frac{|v\times \nabla f|^2}{f^2}, \,  \textcircled{2} \to\frac{|v_* \times \nabla_* f_*|^2}{f_*^2}, \,  \textcircled{3} \to\frac{|\nabla_*f_*|^2}{f_*^2}, \,  \textcircled{4} \to\frac{|\nabla f|^2}{f^2}.
\]
By EDCT~\ref{thm:EDCT}, to show the integral convergence in \eqref{eq:tedct}, it suffices to show, for example,
\[
\iiRsix ff_*|v-v_*|^\gamma \textcircled{1}dvdv_* \to \iiRsix ff_* |v-v_*|^\gamma \frac{|v\times \nabla f|^2}{f^2}dvdv_*,
\]
and similarly for each $\textcircled{i}$ for $i=2,3,4$. By symmetry considerations when swapping the variables $v\leftrightarrow v_*$, the convergence for the terms $\textcircled{1}$ and $\textcircled{4}$ controls the convergence for $\textcircled{2}$ and $\textcircled{3}$, respectively. Hence we will focus on the term $\textcircled{4}$ first and then on term $\textcircled{1}$.

\subsubsection{Term \textcircled{4}}
\label{sec:nocross}
We seek to show in the limit $\epsilon\downarrow 0$,
\begin{align}
	\label{eq:term4converge}
	\begin{split}
		&\quad \iiRsix ff_* |v_*|^2 |v-v_*|^\gamma|b^\epsilon*a^\epsilon|^2 dv_* dv = \iRthree \left(\iRthree f_* |v_*|^2|v-v_*|^\gamma dv_*\right) f|b^\epsilon*a^\epsilon|^2 dv  	\\
		&\to \iRthree \left(\iRthree f_* |v_*|^2|v-v_*|^\gamma dv_*\right) \frac{|\nabla f|^2}{f} dv.
	\end{split}
\end{align}
By the reordering of integrations written above, we now think of the double integral over $v,v_*$ of $ff_* |v_*|^2 |v-v_*|^\gamma |b^\epsilon*a^\epsilon|^2$ as a single integral of the function $\left(\int_{\R^3} f_* |v_*|^2|v-v_*|^\gamma dv_*\right) f|b^\epsilon*a^\epsilon|^2$ over $v$. To be precise, we wish to apply~\Cref{thm:EDCT} with $X = \R^3$ with $H_\epsilon = \left(\int_{\R^3} f_* |v_*|^2|v-v_*|^\gamma dv_*\right) f|b^\epsilon*a^\epsilon|^2$. We can use Cauchy-Schwarz on the convolution integral to absorb the power term as follows
\begin{align*}
	\quad |b^\epsilon*a^\epsilon|^2 = \left|
	\iRthree b^\epsilon(v-w)a^\epsilon(w)dw
	\right|^2 
	&\leq \left(
	\iRthree \japangle{w}^{-\gamma}b^\epsilon(v-w)dw
	\right)\left(
	\iRthree b^\epsilon(v-w)\japangle{w}^{\gamma}|a^\epsilon(w)|^2dw
	\right) 	\\
	&\leq C\japangle{v}^{-\gamma}b^\epsilon*[\japangle{\cdot}^\gamma|a^\epsilon(\cdot)|^2],
\end{align*}
where the last inequality comes from Lemma~\ref{lem:mombeps}. Continuing with Lemma~\ref{lem:Jf}, we have
\begin{equation*}
	\left(
	\iRthree f_* |v_*|^2|v-v_*|^\gamma dv_*
	\right)f|b^\epsilon*a^\epsilon|^2 \leq C f b^\epsilon* [\japangle{\cdot}^\gamma|a^\epsilon|^2].
\end{equation*}
By EDCT~\ref{thm:EDCT}, we reduce the problem to showing in the limit $\epsilon \downarrow 0$
\begin{equation*}
	\iRthree f b^\epsilon*[\langle \cdot \rangle^\gamma|a^\epsilon|^2]dv \to \iRthree \langle v\rangle^\gamma \frac{|\nabla f|^2}{f}dv.
\end{equation*}
This is were we use SACRE, Step 2 of our general strategy~\ref{sec:outlinestrat}. Application of SACRE and further simplification using the specific forms of $a^\epsilon$ and $b^\epsilon$ (see \eqref{eq:aandb}) yields
\begin{equation}
	\label{eq:convparts41}
	\iRthree f b^\epsilon*[\langle \cdot \rangle^\gamma|a^\epsilon|^2]dv = \iRthree [b^\epsilon * f] \japangle{v}^\gamma|a^\epsilon|^2 dv =\iRthree \japangle{v}^\gamma \frac{|b^\epsilon*\nabla f|^2}{b^\epsilon * f}dv.
\end{equation}
We work with this simplified expression and note that pointwise convergence is still valid
\[
\frac{|b^\epsilon* \nabla f|^2}{b^\epsilon * f} \to \frac{|\nabla f|^2}{f}.
\]
Next, we notice that the function $\beta :  (F,f) \mapsto \frac{|F|^2}{f}$ is jointly convex in $F\in \R^3$ and $f>0$, so we can use Jensen's inequality with $b^\epsilon = G^\epsilon$ as the reference probability measure to obtain a further pointwise majorant for the integrand of~\eqref{eq:convparts41}
\begin{align*}
&\quad \frac{|b^\epsilon * \nabla f|^2}{b^\epsilon * f}(v) = \beta(b^\epsilon * \nabla f, \, b^\epsilon * f)(v) = \beta\left(
\int_{\R^3} \nabla f(v-y)b^\epsilon(y)dy, \, \int_{\R^3}f(v-y)b^\epsilon(y)dy
\right) 	\\
&\le \int_{\R^3} \beta(\nabla f(v-y), f(v-y)) b^\epsilon(y)dy = \int_{\R^3} \frac{|\nabla f (v-y)|^2}{f(v-y)}b^\epsilon(y)dy = b^\epsilon * \left[
\frac{|\nabla f|^2}{f}
\right](v).
\end{align*}
Using EDCT~\ref{thm:EDCT} again, we reduce the problem to showing in the limit $\epsilon\downarrow 0$
\begin{equation*}
	\iRthree \langle v\rangle^\gamma b^\epsilon*\left[
	\frac{|\nabla f|^2}{f}
	\right] dv \to \iRthree \langle v\rangle^\gamma \frac{|\nabla f|^2}{f}dv.
\end{equation*}
We use SACRE once more and place the convolution onto the weight term
\begin{equation*}
	\iRthree \langle v\rangle^\gamma b^\epsilon*\left[
	\frac{|\nabla f|^2}{f}
	\right] dv = \iRthree [b^\epsilon* \langle \cdot \rangle^\gamma] \frac{|\nabla f|^2}{f} dv.
\end{equation*}
Now, we are in a position to apply the classical Dominated Convergence Theorem. We notice that we have the pointwise convergence
\[
[b^\epsilon* \langle \cdot \rangle^\gamma]\to \langle v\rangle^\gamma.
\]
Furthermore, using Lemma~\ref{lem:mombeps}, we can estimate $b^\epsilon*\langle \cdot \rangle^\gamma$ uniformly in $\epsilon$ to find the domination
\begin{equation*}
	[b^\epsilon* \langle \cdot \rangle^\gamma] \frac{|\nabla f|^2}{f} \leq C \langle v\rangle^\gamma \frac{|\nabla f|^2}{f}.
\end{equation*}
Using Theorem~\ref{thm:desvillettes16}, the finite entropy-dissipation assumption~\ref{itm:ass3}, and uniformly bounded entropy~\ref{itm:ass2} (remember the constant in~\Cref{thm:desvillettes16} depends also on bounds for the entropy) we know that the right-hand side belongs to $L_v^1$ a.e. $t\in (0,T)$. Therefore, for a.e. $t\in (0,T)$ the conditions of the Dominated Convergence Theorem are satisfied so we have the integral convergence
\[
\iRthree[b^\epsilon* \langle \cdot \rangle^\gamma] \frac{|\nabla f|^2}{f}dv \to \iRthree\langle v\rangle^\gamma \frac{|\nabla f|^2}{f}dv.
\]
We have closed the argument for the convergence of~\eqref{eq:term4converge} after retracing the previous estimates with EDCT~\ref{thm:EDCT}.

\subsubsection{Term \textcircled{1}}
\label{sec:cross}
We seek to show in the limit $\epsilon\downarrow 0$, 
\begin{align}
	\label{eq:term1converge}
	\begin{split}
		&\quad \iiRsix ff_*|v-v_*|^\gamma |v\times (b^\epsilon*a^\epsilon)|^2 dv_* dv = \iRthree \left(
		\iRthree f_* |v-v_*|^\gamma dv_*
		\right)f|v\times (b^\epsilon * a^\epsilon)|^2 dv 	\\
		&\to \iRthree \left(
		\iRthree f_* |v-v_*|^\gamma dv_*
		\right)\frac{|v\times \nabla f|^2}{f} dv 
	\end{split}
\end{align}
using the same strategy of nested applications of EDCT~\ref{thm:EDCT} like in the previous Section~\ref{sec:nocross}. We will encounter difficulty when trying to use Jensen's inequality due to the cross Fisher information term. As in the previous Section~\ref{sec:nocross}, we have written this double integral over $v,v_*$ as a single integral over $v$. By EDCT~\ref{thm:EDCT} and Lemma~\ref{lem:Jf}, it suffices to show the integral convergence of
\begin{equation}
	\label{eq:term11}
	\iRthree \langle v\rangle^\gamma f|v\times (b^\epsilon*a^\epsilon)|^2 dv \to \iRthree \langle v\rangle^\gamma\frac{|v\times \nabla f|^2}{f}
\end{equation}
to obtain the integral convergence of~\eqref{eq:term1converge}. Pointwise, we can make the following manipulations 
\begin{align*}
	v\times (b^\epsilon*a^\epsilon) 
	& = 
	v\times \left(\iRthree G^\epsilon(v-w)\nabla \log (f*G^\epsilon(w))dw\right)
	\\
	& = 
	v\times \left(\iRthree \nabla G^\epsilon(v-w) \log (f*G^\epsilon(w))dw\right)
	\\
	&=
	\iRthree w\times \nabla G^\epsilon(v-w) \log (f*G^\epsilon(w) )dw\\
	&= 
	\iRthree G^\epsilon(v-w) w\times \nabla \log (f*G^\epsilon (w))dw,\addtocounter{equation}{1}\tag{\theequation} \label{eq:changevforw}
\end{align*}
where we have used the radial symmetry of $G^\epsilon$ to get the cancellation $(v-w)\times \nabla G^\epsilon(v-w)=0$ and the twisted integration by parts Lemma~\ref{lem:crossibp} (we note that we not pick any signs in the integration by parts, as the variable $w$ appears with a minus sign in the arguments of  $G^\epsilon$).

We apply Cauchy-Schwarz, multiply and divide by $\japangle{w}^\gamma$, and use Lemma~\ref{lem:mombeps} to obtain
\begin{align*}
	|v\times (b^\epsilon * a^\epsilon)|^2&\leq \left(
	\iRthree G^\epsilon(v-w) \japangle{w}^{-\gamma}dw
	\right)\left(
	\iRthree G^\epsilon(v-w) \japangle{w}^\gamma \left|
	w\times \frac{\nabla f*G^\epsilon(w)}{f*G^\epsilon(w)}
	\right|^2dw
	\right) 	\\
	&\lesssim_\gamma \japangle{v}^{-\gamma}\left(
	\iRthree G^\epsilon(v-w) \japangle{w}^\gamma \left|
	w\times \frac{\nabla f*G^\epsilon(w)}{f*G^\epsilon(w)}
	\right|^2dw
	\right).
\end{align*}
Remembering that this majorant holds pointwise on the integrand of~\eqref{eq:term11}, we multiply by $\japangle{v}^\gamma f(v)$ and obtain
\[
\japangle{v}^\gamma f(v) |v\times (b^\epsilon*a^\epsilon)|^2 \lesssim f\left(
\iRthree G^\epsilon(v-w) \japangle{w}^\gamma \left|
w\times \frac{\nabla f*G^\epsilon(w)}{f*G^\epsilon(w)}
\right|^2dw
\right).
\] 
Now, we recognise a convolution inside the brackets. Hence, using SACRE we can re-write
\[
\iRthree f\left(
\iRthree G^\epsilon(v-w) \japangle{w}^\gamma \left|
w\times \frac{\nabla f*G^\epsilon(w)}{f*G^\epsilon(w)}
\right|^2dw
\right)dv = \iRthree \japangle{v}^\gamma \frac{|v\times \nabla f*G^\epsilon(v)|^2}{f*G^\epsilon(v)}dv.
\]
Using EDCT~\ref{thm:EDCT}, we need to show the convergence of the right-hand side.
Here, it is now possible to use Jensen's inequality after some more manipulations.
\begin{claim}
	\label{claim:Jensencross}
	\begin{equation}
		\label{eq:Jensencross}
		\frac{|v\times \nabla f*G^\epsilon (v)|^2}{f*G^\epsilon(v)} \leq \iRthree G^\epsilon(v-w) \frac{|w\times \nabla f(w)|^2}{f(w)}dw.
	\end{equation}
\end{claim}
\begin{proof}[Proof of Claim~\ref{claim:Jensencross}] We start by repeating a similar argument to \eqref{eq:changevforw}. Using that $G^\epsilon$ is radially symmetric and the twisted integration by parts Lemma~\ref{lem:crossibp} we obtain
	\begin{align*}
		v\times \nabla f*G^\epsilon(v) &= v\times \left(\iRthree \nabla G^\epsilon(v-w)f(w) dw\right)	\\
		&=\iRthree w\times \nabla G^\epsilon(v-w) f(w) dw	\\
		&= \iRthree G^\epsilon(v-w) \underbrace{(w\times \nabla_w f(w))}_{=: F(w)}dw.
	\end{align*}
	Therefore, since $\beta : (F,f) \mapsto \frac{|F|^2}{f}$ is jointly convex in $F\in \R^3$ and $f>0$, we apply Jensen's inequality with $G^\epsilon$ as the reference probability measure to the left-hand side of~\eqref{eq:Jensencross} to see
	\begin{align*}
	&\quad \frac{|v\times \nabla f*G^\epsilon (v)|^2}{f*G^\epsilon(v)} = \frac{|F*G^\epsilon|^2}{f*G^\epsilon}(v) = \beta(G^\epsilon * F, \, G^\epsilon * f)(v) 	\\
	&= \beta\left(
	\int_{\R^3}F(v-w)G^\epsilon(w)dw, \, \int_{\R^3}f(v-w)G^\epsilon(w)dw
	\right) 	\\
	&\le \int_{\R^3} \beta(F(v-w), \, f(v-w)) G^\epsilon(w) dw = \int_{\R^3} \frac{|(v-w)\times \nabla F(v-w)|^2}{f(v-w)} G^\epsilon(w)dw,
	\end{align*}
	which proves the claim.
\end{proof}
Continuing, by EDCT~\ref{thm:EDCT}, we seek to establish the integral convergence of 
\[
\iRthree \japangle{v}^\gamma \left[
\frac{|F|^2}{f} * G^\epsilon
\right](v) dv = \iRthree [\japangle{\cdot}^\gamma * G^\epsilon](v) \frac{|v\times \nabla f(v)|^2}{f(v)}dv.
\]
Finally, the integrand of the right-hand side has a majorant due to Lemma~\ref{lem:mombeps}
\[
[\japangle{\cdot}^\gamma * G^\epsilon](v) \frac{|v\times \nabla f(v)|^2}{f(v)} \lesssim \japangle{v}^\gamma \frac{|v\times \nabla f(v)|^2}{f(v)}.
\]
Once again, using Theorem~\ref{thm:desvillettes16} and  Assumptions~\ref{itm:ass3} and~\ref{itm:ass2}, we obtain that for a.e. $t\in(0,T)$ the right hand side belongs to $L_v^1(\Rthree)$. Using Dominated Convergence theorem, we see that the integral converges. Tracing back the estimates, this takes care of the convergence of the term \textcircled{1} and establishes the convergence in~\eqref{eq:term11}.

We note that the estimates in the previous subsections not only establish the a.e. pointwise convergence of~\eqref{eq:tedct}, but also the majorisation 
$$
\iiRsix \frac{1}{2}ff_* \left|
\tn \left[
\frac{\delta \mathcal{H}_{\epsilon}}{\delta \mu}
\right]
\right|^2 dv_*dv    \leq C\iiRsix \frac{1}{2}ff_* \left|
\tn \left[
\frac{\delta \mathcal{H}}{\delta \mu}
\right]
\right|^2dv_*dv, \quad \text{a.e. }t \quad \forall \epsilon>0,
$$
where
$$
C\lesssim ||\japangle{\cdot}^{-\gamma}f(t)||_{L^\infty\left(0,T;L^1\cap L^\frac{3-\eta}{3+\gamma-\eta}(\Rthree)\right)}+||\japangle{\cdot}^{2-\gamma}f(t)||_{L^\infty\left(0,T;L^1\cap L^\frac{3-\eta}{3+\gamma-\eta}(\Rthree)\right)}
$$
by Lemma~\ref{lem:Jf}. Hence, using assumption~\ref{itm:ass3} and \eqref{eq:tedct} we can apply Lebesgue DCT to pass to the limit in the time integral and show the desired chain rule Claim~\ref{claim:fullchainrule}.

\begin{appendix}
	\section{An auxiliary PDE for Lemma~\ref{lem:Depsslope}}
	\label{sec:auxpde}
	In this section, we fix $\epsilon>0$ throughout and study weak solutions to the following PDE
	\begin{equation}
		\label{eq:auxpde}
		\left\{
		\begin{array}{rcl}
			\partial_t \mu 	&= 	&\nabla\cdot \{
			\mu\phi_{R_1} \iR  \phi_{R_1*} \psi_{R_2}(v-v_*)|v-v_*|^{\gamma+2} \Pi[v-v_*](J_0^\epsilon - J_{0*}^\epsilon)d\mu(v_*)
			\} 	\\
			\mu(0) 	&= 	&\mu_0
		\end{array}
		\right..
	\end{equation}
	We assume the initial data $\mu_0$ belongs to $\mathscr{P}_2(\Rd)$.
	For $R_1,R_2 >0$, the functions  $0 \leq \phi_{R_1}, \psi_{R_2} \leq 1$ are smooth cut-off functions used to approximate the identity function in different ways.
	\[
	\phi_{R_1}(v) = \left\{
	\begin{array}{cl}
		1, 	&|v| \leq R_1 	\\
		0, 	&|v| \geq R_1+1
	\end{array}
	\right., 	\quad
	\psi_{R_2}(z) = \left\{
	\begin{array}{cl}
		0, 	&|z| \leq 1/R_2 	\\
		1, 	&|z| \geq 2/R_2
	\end{array}
	\right..
	\]
	For $\epsilon > 0$, $J_0^\epsilon$ is the gradient of first variation of $\mathcal{H}_{\epsilon}$ \textit{applied to }$\mu_0$, meaning
	\[
	J_0^\epsilon = \nabla G^{\epsilon} * \log [\mu_0 * G^{\epsilon}]\in C^\infty(\Rd;\Rd).
	\]
	The main result of this section is
	\begin{theorem}
		\label{thm:auxpde}
		Fix $\epsilon,R_1,R_2>0$, $\gamma\in \R$, and $\mu_0 \in \mathscr{P}_2(\Rd)$. Then, there is a global unique weak solution $\mu\in C([0,+\infty); \, \mathscr{P}_2(\R^d))$ to~\eqref{eq:auxpde}.
	\end{theorem}
	By Lemma~\ref{lem:estlogdiff}, we know that $J_0^\epsilon$ is uniformly bounded (with constant depending on $\epsilon$ and $\mu_0$ only through bounds on its second moment). The purpose of $\phi_{R_1}, \phi_{R_1*}$ is to cut off the growth of $J_0^\epsilon,J_{0*}^\epsilon$ to ensure that the `velocity field' in the right-hand side of~\eqref{eq:auxpde} is globally Lipschitz (it is, in fact, smooth and compactly supported). The $\psi_{R_2}(v-v_*)$ term avoids the possible singularities coming from the weight $|v-v_*|^{\gamma+2}$ for soft potentials $\gamma < 0$.
	
	The construction of the solution in~\Cref{thm:auxpde} is given in two steps. Firstly, a local well-posedness theory established to some finite time interval $T>0$ which depends on $\epsilon, \, \gamma, \, R_1, \, R_2$ and $\mu_0$. Secondly, the time of existence (and uniqueness) is extended to $+\infty$ since $T$ depends on $\mu_0$ only through its second moment which is conserved by the evolution of~\eqref{eq:auxpde}.
	
	We fix $T>0$ to be determined explicitly later. Our strategy is to employ a fixed point argument in the space $C([0,T];\mathscr{P}_2(\Rd))$ which we will equip with the following metric
	\[
	d(\mu,\nu) := \sup_{t\in [0,T]}W_2(\mu(t),\nu(t)), \quad \mu,\nu \in C([0,T];\mathscr{P}_2(\Rd)),
	\]
	where $W_2$ is the 2-Wasserstein distance on $\mathscr{P}_2(\Rd)$. We have closely followed the procedure in~\cite{canizo_well-posedness_2011} with appropriate modifications for this setting.
	\begin{remark}
		Since we are cutting off the `velocity' field at radius $R_1,R_2$, the growth of $J_0^\epsilon$ is inconsequential. Hence the results of this section can be applied when replacing the convolution kernel of $J_0^\epsilon$ with general tailed exponential distributions $G^{s,\epsilon}(v)$ for $s>0$.
	\end{remark}
	For $\mu\in\mathscr{P}_2(\Rd)$, we will denote by $U[\mu](v)$ the following function
	\[
	U[\mu](v) := -\phi_{R_1} \iR  \phi_{R_1*} \psi_{R_2}(v-v_*)|v-v_*|^{\gamma+2} \Pi[v-v_*](J_0^\epsilon - J_{0*}^\epsilon)d\mu(v_*),
	\]
	so that the PDE in~\eqref{eq:auxpde} can be written as a nonlinear transport/continuity equation
	\[
	\partial_t \mu(t) = -\nabla\cdot \left\{
	\mu(t) U[\mu(t)]
	\right\}.
	\]
	To fix ideas, the weak formulation of~\eqref{eq:auxpde} is such that the following equality holds for all test functions $\tau\in C_c^\infty(\Rd)$ and times $t\in[0,T]$
	\begin{align*}
		&\iR\tau(v)d\mu_r(v) - \iR\tau(v)d\mu_0(v) 	\\
		&\quad = \int_0^t \iR \phi_{R_1}\nabla\tau(v) \cdot \iR \phi_{R_1*} \psi_{R_2}(v-v_*)|v-v_*|^{\gamma+2} \Pi[v-v_*](J_0^\epsilon - J_{0*}^\epsilon)d\mu_s(v_*)d\mu_s(v)ds.
	\end{align*}
	Thanks to all the smooth cutoffs from $\phi_{R_1},\phi_{R_1*},$ and $\psi_{R_2}$ and $\mu_0\in\mathscr{P}_2(\Rd)$, we can enlarge the class of test functions to smooth functions with quadratic growth. In particular, by choosing $\tau(v) = |v|^2$ and symmetrising the right-hand side by swapping $v\leftrightarrow v_*$, we see that the second moment of $\mu_0$ is conserved along the evolution of~\eqref{eq:auxpde}.
	
	Our first step is to look at the level of the characteristic equation associated to~\eqref{eq:auxpde}.
	\begin{lemma}
		[Characteristic equation]
		\label{lem:charaux}
		For any $T>0$, $\mu\in C([0,T];\mathscr{P}_2(\Rd))$ and $v_0\in\Rd$, there exists a unique solution $v\in C^1((0,T);\Rd)\cap C([0,T];\Rd)$ to the following ODE
		\[
		\frac{dv}{dt} = U[\mu(t)](v), \quad v(0) = v_0.
		\]
		Furthermore, the growth rate satisfies
		\[
		|v(t)| \leq \max\{
		|v_0|,R_1+1
		\}, \quad \forall t\in[0,T].
		\]
	\end{lemma}
	\begin{proof}
		$U[\mu(t)](\cdot)$ is smooth and compactly supported uniformly in $t$, so classical Cauchy-Lipschitz theory gives existence and uniqueness of solution $v$ with the promised regularity.
		
		For the estimate on the growth rate, note that $U[\mu]$ has support contained in $B_{R_1+1}$. Points outside this ball do not change in time according to this ODE.
	\end{proof}
	We will denote by $\Phi_{\mu}^t$ the flow map associated to this ODE, so that
	\[
	\frac{d}{dt}\Phi_{\mu}^t(v_0) = U[\mu(t)](\Phi_{\mu}^t(v_0)), \quad \Phi_{\mu}^0(v_0) = v_0.
	\]
	It is known that, given $\nu\in C([0,T];\mathscr{P}_2(\Rd))$, the curve of probability measures $\mu(t) = \Phi_{\nu}^t \# \mu_0$ is a weak solution to
	\[
	\partial_t \mu(t) = -\nabla\cdot \left\{
	\mu(t) U[\nu(t)]
	\right\}, \quad \mu(0) = \mu_0.
	\]
	Here, $\Phi_{\nu}^t\# \mu_0$ is the push-forward measure of $\mu_0$ defined in duality with $\tau\in C_b(\Rd)$ by
	\[
	\iR \tau(v) d(\Phi_{\nu}^t\# \mu_0)(v) = \iR \tau(\Phi_{\nu}^t(v))d\mu_0(v).
	\]
	We seek to find a fixed point to the map $\mu\mapsto \Phi_\mu^t \# \mu_0$ as it would weakly solve~\eqref{eq:auxpde}. To better understand the properties of this map, we need to establish estimates on the flow map through $U$ as a function of time and measures.
	\begin{lemma}
		[$L^\infty$ estimate for velocity field]
		\label{lem:Linftyvel}
		There exists a constant $C = C(\epsilon,\gamma, R_1,R_2, \mu_0)>0$  such that for every $T>0$ and $\nu \in C([0,T];\mathscr{P}_2(\Rd))$, we have
		\[
		|U[\nu(t)](v)| \leq C, \quad \forall t\in[0,T], v\in\Rd.
		\]
	\end{lemma}
	\begin{proof}
		\underline{Estimate for $\gamma \geq -2$:}
		\\
		We have the following three inequalities
		\[
		|v-v_*|^{\gamma+2} \lesssim_\gamma |v|^{\gamma+2} + |v_*|^{\gamma+2}, \quad ||\Pi[v-v_*]||\leq 1, \quad J_0^\epsilon \lesssim_{\epsilon,\mu_0}1
		\]
		due to the range of $\gamma$, boundedness of $\Pi$, and Lemma~\ref{lem:estlogdiff}, respectively. These three inequalities provide the estimate
		\[
		|U[\nu(t)](v)| \lesssim_{\gamma,\epsilon,\mu_0} \phi_{R_1}(v) \iR \phi_{R_1}(v_*)(|v|^{\gamma+2} + |v_*|^{\gamma+2})d\nu_t(v_*),
		\]
		where we have dropped $\psi_{R_2}$ altogether. For the integral term, we apply H\"older's inequality taking advantage of the compact support of $\phi_{R_1}$ and the unit mass of $\nu_t$ to further obtain
		\[
		|U[\nu(t)](v)| \lesssim_{\gamma,\epsilon,\mu_0}\phi_{R_1}(v)(R_1^{2+\gamma} + \langle v \rangle^{2+\gamma}) \iR d\nu_t(v_*) \lesssim_{R_1} \phi_{R_1}(v) \langle v \rangle^{2+\gamma}.
		\]
		Again, since $\phi_{R_1}$ has compact support, we can brutally estimate the polynomial to conclude.
		\\
		\underline{Estimate for $\gamma < -2$:}
		\\
		Unlike the previous case, we change one of the inequalities due to the unavailability of a triangle inequality and use
		\[
		\psi_{R_2}(v-v_*)|v-v_*|^{\gamma+2} \lesssim 1/R_2^{\gamma+2}, \quad ||\Pi[v-v_*]||\leq 1, \quad J_0^\epsilon \lesssim_{\epsilon,\mu_0}1.
		\]
		From these inequalities and the compact support of $\phi_{R_1}$, we have
		\[
		|U[\nu(t)](v)| \lesssim_{\gamma,\epsilon,\mu_0,R_2} \phi_{R_1}(v) \iR \phi_{R_1}(v_*) d\nu_t(v_*) \leq 1,
		\]
		which concludes the proof.
	\end{proof}
	The next result follows exactly as in~\cite{canizo_well-posedness_2011}.
	\begin{lemma}
		[Time continuity of flow map]
		\label{lem:timecontflow}
		Let $C=C(\epsilon,\gamma,R_1,R_2,\mu_0)>0$ be the same constant from Lemma~\ref{lem:Linftyvel}. Then for any $T>0$, and $\nu\in C([0,T];\mathscr{P}_2(\Rd))$ we have
		\[
		||\Phi_{\nu}^t - \Phi_\nu^s||_{L^\infty(\Rd)} \leq C|t-s|.
		\]
	\end{lemma}
	Our next objective is to establish the regularity of the flow map with respect to the measures in the subscript. To simplify the subsequent lemmata, let us use the notation in the following
	\begin{lemma}
		\label{lem:defF}
		Define 
		\[
		F: (v,w) \in \Rd\times \Rd \mapsto
		\phi_{R_1}(v) \phi_{R_1}(w) \psi_{R_2}(v-w)|v-w|^{\gamma+2}\Pi[v-w](J_0^\epsilon(v) - J_0^\epsilon(w)). 	
		\]
		The function $F$ is smooth and compactly supported. In particular, for every $k,l\in\mathbb{N}$, there is a constant $C=C(\epsilon,\gamma,R_1,R_2,\mu_0,k,l)>0$ such that
		\[
		||D_v^kD_w^l F||_{L^\infty(\Rd\times \Rd)} \leq C.
		\]
		More precisely, the constant $C$ depends on $\mu_0$ only through bounds on its second moment as in~\Cref{lem:estlogdiff}.
	\end{lemma}
	\begin{proof}
		The compact support property comes from the factor of $\phi_{R_1}(v)\phi_{R_1}(w)$ in the definition. The regularity comes from the avoidance of $v=w$ due to the factor $\psi_{R_2}(v-w)$.
	\end{proof}
	\begin{corollary}
		[Pointwise and measurewise regularity of $U$]
		\label{cor:smoothU}
		Consider the constant $C =C(\epsilon,\gamma,R_1,R_2,\mu_0,k,l)>0$ from Lemma~\ref{lem:defF} above. We have the following
		\begin{enumerate}
			\item Take $C_1=C(\epsilon,\gamma,R_1,R_2,\mu_0,0,1)>0$. For every $T>0; \nu^1,\nu^2 \in C([0,T];\mathscr{P}_2(\Rd)); t\in[0,T]; v\in\Rd$ we have the estimate
			\[
			|U[\nu^1(t)](v) - U[\nu^2(t)](v)| \leq C_1W_2(\nu_t^1,\nu_t^2).
			\]\label{item:point1U}
			\item Take $C_2 = C(\epsilon,\gamma,R_1,R_2,\mu_0,1,0) >0$. For every $T>0;\nu\in C([0,T];\mathscr{P}_2(\Rd));t\in[0,T];v_1,v_2\in\Rd$ we have the estimate
			\[
			|U[\nu(t)](v_1) - U[\nu(t)](v_2)| \leq C_2|v_1-v_2|.
			\]\label{item:point2U}
		\end{enumerate}
	\end{corollary}
	\begin{remark}
		By considering the anti-symmetric property of $F$ when swapping variables $v\leftrightarrow w$, one really obtains $C_1 = C_2$. Their distinction in this corollary is artificial.
	\end{remark}
	\begin{proof}
		\underline{Item~\ref{item:point1U}:}
		\\
		Firstly, for every $t\in[0,T]$ take $\pi(t)\in\mathscr{P}_2(\Rd\times \Rd)$ the 2-Wasserstein optimal transportation plan connecting $\nu^1(t)$ and $\nu^2(t)$ which exists, see~\cite{V09}. We estimate the difference with notation from Lemma~\ref{lem:defF}
		\begin{align*}
			|U[\nu^1(t)](v) - U[\nu^2(t)](v)|  	&= \left|
			\iR F(v,w)d\nu_t^1(w) - \iR F(v,\bar{w})d\nu_t^2(\bar{w})
			\right| 	\\
			&= \left|
			\iiR F(v,w) - F(v,\bar{w})d\pi_t(w,\bar{w})
			\right| 	\\
			&\leq C_1\iiR |w-\bar{w}| d\pi_t(w,\bar{w}) 	\\
			&\leq C_1 W_2(\nu_t^1,\nu_t^2).
		\end{align*}
		The first inequality uses a mean-value type estimate (in the second variable of $F$) and the second inequality uses Cauchy-Schwarz or equivalently, that $W_2$ is stronger than $W_1$.
		\\
		\underline{Item~\ref{item:point2U}:}
		\\
		As with item~\ref{item:point1U}, we estimate the difference using $F$ to find
		\begin{align*}
			|U[\nu(t)](v_1) - U[\nu(t)](v_2)| 	&= \left|
			\iR F(v_1,w) - F(v_2,w)d\nu_t(w)
			\right| 	\\
			&\leq \iR |F(v_1,w) - F(v_2,w)| d\nu_t(w) 	\\
			&\leq C_2 |v_1-v_2|.
		\end{align*}
		Once more, a mean-value type estimate is applied (in the first variable of $F$) and we recall $\nu_t$ is a probability measure.
	\end{proof}
	The next result combines both items of Corollary~\ref{cor:smoothU} to estimate the regularity of the flow map with respect to measures and follows exactly as in~\cite{canizo_well-posedness_2011}.
	\begin{lemma}
		[Continuity of flow map with respect to measures]
		\label{lem:meascontflow}
		For $T>0$ fix any $\nu^1,\nu^2\in C([0,T];\mathscr{P}_2(\Rd))$ and  $t\in[0,T]$. With $C:= C_1 = C_2$ the same constants in Corollary~\ref{cor:smoothU}, we have the estimate
		\[
		||\Phi_{\nu^1}^t - \Phi_{\nu^2}^t||_{L^\infty(\Rd)} \leq (e^{Ct}-1)d(\nu^1,\nu^2),
		\]
		recalling that $d(\nu^1,\nu^2) = \sup_{t\in [0,T]} W_2(\nu_t^1,\nu_t^2)$.
	\end{lemma}
	It is by now classical how to obtain Theorem~\ref{thm:auxpde} from Corollary~\ref{cor:smoothU} and Lemma~\ref{lem:meascontflow}, see~\cite{canizo_well-posedness_2011,CCH14,Golse16} for instance. The time of existence can be given by any $0 < T < \frac{1}{C}\log 2$ where $C>0$ is chosen as in Lemma~\ref{lem:meascontflow} and the result follows by a fixed point argument. The extension to all times is owed to the fact that $C>0$ depends on the initial data $\mu_0$ only through its second moment. This quantity is conserved through by the evolution of~\eqref{eq:auxpde} and so the maximal time of existence is $+\infty$.
\end{appendix}

\section*{Acknowledgements}
JAC was supported the Advanced Grant Nonlocal-CPD (Nonlocal PDEs for Complex Particle Dynamics: 	Phase Transitions, Patterns and Synchronization) of the European Research Council Executive Agency (ERC) under the European Union's Horizon 2020 research and innovation programme (grant agreement No. 883363). JAC and MGD were partially supported by EPSRC grant number EP/P031587/1. MGD was partially supported by CNPq-Brazil (\#308800/2019-2) and Instituto Serrapilheira. JW was funded by the President's PhD Scholarship program of Imperial College London. JAC and JW were also partially supported by the Royal Society through the International Exchange Scheme 2016 CNRS France. JAC and MGD would like to thank the American Institute of Mathematics since our attendance to the AIM workshop ``Nonlocal differential equations
in collective behavior'' in June 2018 triggered this research.

\bibliographystyle{abbrv}
\bibliography{My_Library}
\end{document}